\documentclass[10pt,twoside,a4paper,reqno]{amsart}
\usepackage{fix-cm}
\usepackage[T1]{fontenc}

\usepackage{enumerate,url,amsfonts,amsthm,amsmath,amssymb,amscd, color}
\usepackage{hyperref}
\usepackage[abbrev]{amsrefs}
\usepackage{blkarray}
\usepackage{tikz}
 \usetikzlibrary{calc} 

\theoremstyle{definition}
\newtheorem{theo+}           {Theorem}
\newtheorem{prop+}[theo+]    {Proposition}
\newtheorem{coro+}[theo+]            {Corollary}
\newtheorem{lemm+}[theo+]            {Lemma}
\newtheorem{conjecture}[theo+]       {Conjecture}

\newtheorem{defi+}           {Definition}
\newtheorem{problem}         {Problem}

\newtheorem*{ack}            {Acknowledgement}

\newtheorem*{not+}            {Notation}
\newtheorem{rema+}           {Remark}
\newtheorem{example}         {Example}

\newenvironment{theorem}{\begin{theo+}}{\end{theo+}}
\newenvironment{proposition}{\begin{prop+}}{\end{prop+}}
\newenvironment{corollary}{\begin{coro+}}{\end{coro+}}
\newenvironment{lemma}{\begin{lemm+}}{\end{lemm+}}
\newenvironment{remark}{\begin{rema+}}{\end{rema+}}

\newcommand{\coorshifting}{
\coordinate (AA) at ($(A)+(S1)$);
\coordinate (AAp) at ($(Ap)+(S1)$);
\coordinate (BB) at ($(B)+(S1)$);
\coordinate (BBp) at ($(Bp)+(S1)$);
\coordinate (CC) at ($(C)+(S1)$);
\coordinate (CCp) at ($(Cp)+(S1)$);
}
\newcommand{\lan}{\langle}
\newcommand{\ran}{\rangle}

\newcommand{\bL}{\mathbb {L}}

\newcommand{\bR}{\mathbb R}
\newcommand{\RR}{\mathbb R}
\newcommand{\bZ}{\mathbb Z}

\newcommand{\Verti}{\mathcal{V}}
\newcommand{\Vol}{\mathrm{Vol}}
\newcommand{\conv}{\mathrm{conv}}
\newcommand{\Boxe}{\text{Box}_\varepsilon}

\newcommand{\B}{\mathcal B}

\newcommand {\cP} {\mathcal P}
\newcommand{\cS}{\mathfrak S}
\newcommand{\cT}{\mathcal{T}}

\newcommand{\De}{\Delta}
\newcommand{\eps}{\epsilon}
\newcommand{\veps}{\varepsilon}

\newcommand{\U}{\mathbf u}
\newcommand{\bfe}{\mathbf e}
\newcommand{\bi}{\mathbf i}
\newcommand{\bm}{\mathbf m}

\newcommand{\bu}{\mathbf u}
\newcommand{\bv}{\mathbf v}
\newcommand{\bw}{\mathbf w}
\newcommand{\bW}{\mathbf W}
\newcommand{\bx}{\mathbf x}
\newcommand{\x}{\mathbf {x}}
\newcommand{\z}{\mathbf {z}}

\newcommand {\LL}{\mathfrak L}

\newcommand{\M}{\mathfrak M}
\newcommand{\F}{\mathfrak F}
\newcommand{\Rat}{\mathfrak {Rat}}

\numberwithin{equation}{section}

\begin{document}

\title[On  moments of a polytope]
{On  moments of a polytope}

\author [N.~Gravin] {Nick Gravin}

\address {Shanghai University of Finance and Economics, 
100 Wudong Road, Yangpu district, Shanghai, China}
\email{nikolai@mail.shufe.edu.cn}

\author [D.V.~Pasechnik] {Dmitrii V.  Pasechnik}

\address {Department of Computer Science,
University of Oxford, Wolfson Building, Parks Road, Oxford OX1 3QD, UK}
\email{dimpase@cs.ox.ac.uk}

\author[B.~Shapiro]{Boris Shapiro}
\address{Department of Mathematics, Stockholm University, SE-106 91 Stockholm,
      Sweden}
\email{shapiro@math.su.se}

\noindent
\author[M.~Shapiro]{Michael Shapiro}
\noindent
\address{Department of Mathematics, Michigan State University, East
Lansing, MI 48824-1027, USA}
\email{mshapiro@math.msu.edu}

\date{\today}


\begin{abstract}  
We show that the multivariate generating function of
appropriately normalized moments of a measure with homogeneous polynomial
density supported on a compact polytope $\cP\subset \bR^d$ 
is a rational function.  Its denominator is the product
of linear forms dual to the vertices of $\cP$ raised  to the power
equal to the degree of the density function. 
 %The used normalization  of moments is connected with a special case of  the Fantappi\`{e}  transform of the weight. 
Using this, we solve the inverse moment problem for the set of, not necessarily convex, 
polytopes having a given  set $S$ of vertices. 
Under a weak non-degeneracy assumption we also show that the uniform measure supported on any such 
polytope is a linear combination of uniform measures supported on simplices with vertices in $S$.
\end{abstract}

\dedicatory{To the memory of Mikael Passare}

\maketitle

\section{Introduction}\label{s1}

The initial motivation for the present paper came from proposed in \cite{GLPR} efficient algorithm
recovering an arbitrary  convex polytope from axial moments of a polynomial measure supported on it. 
This algorithm is based on  the formulas for the axial moments of polytopes found over 20 years ago independently 
by M.~Brion, J.~Lawrence, A.~Khovanskii-A.~Pukhlikov, and A.~Barvinok \cites{Bri,La,MR1190788,Bar2}, 
see \cites{MR2455889,BR} for accessible explanation.  
In  \cite{GLPR}  the authors made an essential, although implicit, 
use of a univariate rational generating function for 
appropriately normalized axial moments.
Here a multivariate, and explicit, analog of the latter function is developed. It turns out it provides a very convenient
encoding of non-convex polytopes, which is of independent interest. E.g. it leads to a natural definition of
\emph{vertices} of such non-convex polytopes, which have similar  properties to vertices of convex polytopes.
It also allows to find the \emph{exact solutions} of a class of inverse moment problems on non-convex polytopes.

After the first version \cite{GPSS12} of this text was released in 2012,
it was pointed out to us by Prof. Mich{\`e}le Vergne
that Laplace transform techniques developed for studying hyperplane 
arrangements in \cite{BrVe99} 
simplify and strengthen a number of
our results. We discuss this in the Section~\ref{subs:arr},
while leaving
full details for another publication.

%(Leaving an interested
%reader to  explore new features of Google Scholar    we do not
%attempt in this short note to give an overview of  the  vast
%field of  the classical inverse moment problem  going back  to H.~Poincar\'e.) The main result
%of the present paper is a simple formula for the multivariate generating
%function of all moments of an arbitrary polytope w.r.t. to an arbitrary
%homogeneous polynomial weight function which generalizes %To the best of our knowledge it is the first
%one of its kind giving all moments for a sufficiently large class of bodies
%and/or weight functions.

\begin{not+}
%Let $\cP\subset \bR^d$
%denote an arbitrary  compact $d$-dimensional polytope not necessarily convex or
%connected.  (We assume that $\bR^d$ is endowed with the standard scalar product
%$\lan  . , . \ran $ and with  fixed orthonormal coordinates  $(x_1,...,x_d)$.)
%Let $\rho(x_1,...,x_d)$ be an arbitrary non-trivial homogeneous polynomial of
%some degree $\delta$. 
In what follows we shall always assume that $\bR^d$ is endowed  with a fixed coordinate system $(x_1,...,x_d)$,  
orthonormal with respect to the standard scalar product $\lan \cdot, \cdot \ran$.  Let $\mu$ be a finite complex-valued Borel measure in $\bR^d$. (For standard measure-theoretic notions we follow \cite{MR924157}.) 
Given a multiindex $I=(i_1,\dots,i_d)$, let $\x^I$ be the shorthand of the 
monomial $x_1^{i_1}\dots x_d^{i_d}$ and $|I|$ the shorthand for $i_1+\dots +i_d$.    
For any multiindex  $I$, define the  {\em moment} $m_I(\mu)$ of $\mu$  as
\begin{equation}\label{eq:moment}
%m_{i_1,...,i_d}^\rho(\cP)
m_I(\mu):=\int_{\bR^d}x_1^{i_1}x_2^{i_2}...x_d^{i_d}d\mu(x_1,x_2,\ldots, x_d)=\int_{\bR^d}{\bx}^I  d\mu(\bx).
\end{equation}
%Given a vector $\Xi=(\xi_1,...,\xi_d)\in \bR^d$ we call  the function $$L_\Xi(u_1,...u_d)=1-\lan \Xi,U\ran =1-\xi_1u_1+...+\xi_du_d$$ the {\em linear form dual to $\Xi$}. (By definition, if $\Xi=0$ then $L_\Xi=1$.)

Define the {\em normalized moment generating function} $F_{\mu}(\U)=F_{\mu}(u_1,\dots,u_d)$ of $\mu$  by 
\begin{equation}\label{eq:nmomgf}
F_{\mu}(\U):=\sum_{I:=(i_1,\dots,i_d)\geq 0}\frac{(|I|+d)!}{i_1!\cdots i_d!} 
m_{I}(\mu) \U^I,\quad\text{where $\U^I=u_1^{i_1}\dots u_d^{i_d}$.}
\end{equation}

Note that $F_{\mu}(\bu)$ admits the  integral representation 
\begin{equation} \label{eq:Mainint}
F_{\mu}(\bu)=d!\int_{\bR^d}\frac{d\mu(\x)}{(1-\lan \x,\U\ran)^{d+1}},
\end{equation}
which is a special case of a \emph{Fantappi\`{e}  transformation}. For details on the latter, see e.g.
\cite{APS}*{Chapter~3}.
 A proof of \eqref{eq:Mainint} will be  given at the end of Section~\ref{s2}; see also Remark~\ref{rem:Fanta}. 

%\begin{rema+}
%The definitions of  $m_I^\rho(\cP)$ and $F_{\cP}^\rho(\U)$ as well as relation~\eqref{eq:Mainint} 
%remain valid for arbitrary compact sets  $\cP\subset\RR^d$. They can be also extended to a larger class of measurable sets but we do not need this extension  for the purposes of the present paper. 
%The same is also true for \eqref{eq:Mainint},
%if $\cP$ is, in addition, compact.
%\end{rema+}

Given any complex-valued finite measure $\mu$ 
and any degree $\delta$ homogeneous $d$-variate polynomial $\rho$, 
it is convenient to define the (re)normalized moment
generating function $F_\mu^\rho(\bu)$ for the measure $\rho\mu$, where
by definition, $\int_{\bR^d}f d(\rho\mu)=\int_{\bR^d} f \rho d\mu$,
in such a way that it can be obtained  from  $F_\mu(u)$ by application of the 
differential operator $\rho\left (\frac{\partial}{\partial \bu}\right)$. Namely, set
\begin{equation}\label{eq:nmomgfrho}
F_{\mu}^\rho(\bu):=\sum_{I:=(i_1,...,i_d) \ge 0}\frac{(|I|+d+\delta)!}{i_1!\cdots i_d!}
m_{I}(\rho\mu) \bu^I.
\end{equation}
Note that $F_{\mu}^\rho(\bu)\neq F_{\rho\mu}(\bu)$ for non-constant $\rho$. 
However, they are also connected, by an 
explicit differential operator as follows. 
\begin{theorem}\label{th:weight} 
For any complex-valued finite measure $\mu$ and any homogeneous polynomial $\rho$ of degree $\delta$, 
\begin{align} \label{eq:gen_i}
F_{\mu}^\rho (\bu)&=
\prod_{\ell=d}^{d+\delta-1}\left(\sum_k u_k\frac{\partial}{\partial u_k}+\ell\right)\circ F_{\rho\mu}(\bu)\\
\label{eq:gen_ii}
&=\rho\left (\frac{\partial}{\partial \bu}\right)\circ F_{\mu}(\bu)\\
\label{eq:gen_iii}
&=(d+\delta)!\int_{\bR^d}\frac{\rho(\x)d\mu(\x)}{(1-\lan \x,\U\ran)^{d+\delta+1}}.
\end{align}
\end{theorem}
\end{not+}
Here and in what follows $\circ$ denotes the application of a differential operator to a function.
The proof of the latter result is basically an exercise in manipulating formal power series, 
and we do not claim its novelty.
For the sake of completeness, we include a proof in Section~\ref{s2}.

\subsection*{Results on convex polytopes.}
A finite set $S\subset \bR^d$ is called {\it spanning} if it is not contained in any (affine)
hyperplane in $\bR^d$.  (Obviously, $\text{card}(S)\ge d+1$.)
As usual, by a (compact, convex) \emph{polytope} $\cP\subset \bR^d$ we  mean the convex hull of a finite 
spanning set in  $\bR^d$.   The set of vertices of a convex polytope $\cP$ is the inclusion-minimal finite set with convex hull $\cP$. %(Obviously, the set of vertices of a convex $\cP$ is uniquely defined.) 
A {\it $d$-simplex} in $\bR^d$ is the convex hull of a spanning $(d+1)$-tuple of  points.  By an {\it open polytope (resp. simplex)} we mean the set of interior points of a compact polytope (resp. simplex).  

Given a convex polytope $\cP$ let $\Verti=(\bv_1,...,\bv_N)$ denote the set of its vertices.  
Assume that $\cP$ is  simple, i.e. each $\bv\in \Verti$ 
has exactly $d$ incident edges $\bv\bv_{e_1}$, \dots, $\bv\bv_{e_d}$. 
Set $w_k(\bv):=\bv_{e_k}-\bv$, for $1\leq k\leq d$.
The non-negative real span $K_{\bv}$ of $w_1(\bv)$,\dots, $w_d(\bv)$  
is called {\em the tangent cone} of $\cP$ 
at $\bv$. For each $K_{\bv},$ define $|\det K_{\bv}|=|\det (w_1(\bv),\ldots ,w_d(\bv))|$ to be the
volume of the parallelepiped formed by $w_1(\bv),\ldots, w_d(\bv)$. 

Given a bounded domain $\Omega\subset\bR^d$, we call the measure 
$$\mu_\Omega=\chi_\Omega dx_1dx_2\ldots dx_d,$$
 where $\chi_\Omega$ is the characteristic function of $\Omega$,  the {\it standard measure} of $\Omega$.  

For a simple convex polytope $\cP$, we have the following explicit representation of $F_{\mu_\cP}(\bu).$
\begin{theorem}\label{th:weight1} For an arbitrary simple convex polytope $\cP$,  
\begin{align}\label{eq:Main}
F_{\cP}(\bu):=F_{\mu_{\cP}}(\bu)&=(-1)^d\sum_{\bv\in \Verti}\frac{\lan \bv,\U\ran ^d|\det K_{\bv} |}{\prod\limits_{j=1}^d\lan w_j(\bv),\U\ran }\cdot \frac{1}{1-\lan \bv,\U\ran }\\
\label{eq:Mainsimp}
&=(-1)^d\sum_{\bv\in \Verti}\frac{|\det K_{\bv} |}{\prod\limits_{j=1}^d\lan w_j(\bv),\U\ran }\cdot \frac{1}{1-\lan \bv,\U\ran }.
\end{align}
\end{theorem}

\medskip
\begin{rema+} Instead of the explicit choice of $w_k(\bv)$ for $\bv\in\Verti$ made above, we can 
take any fixed set of non-zero vectors $w_1(\bv),\ldots,w_d(\bv)$, 
spanning the tangent cone of $\bv$ in $\cP$. This does not affect the validity of \eqref{eq:Main} and
\eqref{eq:Mainsimp}.
\end{rema+}

Theorem~\ref{th:weight1} implies
%F_{\Delta}^1(u_1,...,u_d)=  \frac  {d! \Vol (\Delta)}   {\prod_{i=1}^{d+1} L_{v_i}},
\begin{corollary}\label{cor:simplex}
Let $\Delta=\conv(\Verti)\subset \bR^d$ be an arbitrary  $d$-simplex. Then 
\begin{equation}\label{eq:simplex}
F_{\Delta}(\bu)=\frac{d!\Vol(\Delta)}{\prod\limits_{\bv\in\Verti }(1-\lan\bv,\bu\ran)}.
\end{equation}
%where $L_{v_1},\ldots , L_{v_{d+1}}$ are the linear forms dual to the vertices $v_1,\ldots , v_{d+1}$ of the simplex $\Delta$.
\end{corollary}

\medskip
\begin{rema+} As we discovered after we proved the above results,  statements similar to   Corollary~\ref{cor:simplex}   in the complex setting can be found  in \cite{APS}*{Section 3.5}
and in particular \cite{APS}*{Corollary 3.5.6}.

A variation of \eqref{eq:simplex} also appears in \cite{BBDL}, in the context of
designing an efficient procedure for integration of polynomials
over simplices.  
\end{rema+}

Notice that an arbitrary convex polytope $\cP$ admits a triangulation which only uses
the existing vertices of $\cP$, see e.g. \cite{BR}*{Theorem~3.1}.  Applying
Corollary~\ref{cor:simplex} and Theorem~\ref{th:weight} to the sum of measures
corresponding to such a triangulation we get the following.
\begin{corollary}\label{cor:arbit}
The normalized moment generating function $F_{\cP}^\rho(\bu)$ 
of any convex polytope $\cP$ with respect to  any homogeneous
polynomial density function $\rho$ of degree $\delta$ is a rational function with denominator 
dividing
%is the product of all the linear forms dual to its vertices raised to the power $\delta$.
$$\prod_{\bv\in\Verti}(1-\lan \bv,\bu\ran)^\delta.$$
\end{corollary}

\begin{example}\label{ex:one} 
Let $\Delta$ be a triangle in $\bR^2$ with vertices $v_1=(1,1),\;v_2=(2,5)$ and $v_3=(3,2)$. 
Its normalized moment generating function equals
$$F_\Delta(u_1,u_2)=\frac{7}{(1-u_1-u_2)(1-2u_1-5u_2)(1-3u_1-2u_2)}.$$
Its Taylor expansion about the origin up to the terms of degree $7$ is given by
\begingroup
\everymath{\scriptstyle}
\scriptsize
\begin{multline*}
7+42u_1+56u_2+175u_1^2+455u_1u_2+329u_2^2+630u_1^3+2387u_1^2u_2+3367u_1u_2^2+1750u_2^3+2107u_1^4+10318u_1^3u_2\\
+21217u_1^2u_2^2+21546u_1u_2^3+8967u_2^4+6762u_1^5+40082u_1^4u_2+106526u_1^3u_2^2+157976u_1^2u_2^3+128772u_1u_2^4\\+45276u_2^5+21175u_1^6+145845u_1^5u_2+468895u_1^4u_2^2+900123u_1^3u_2^3+10744451u_1^2u_2^4+741993u_1u_2^5+227269u_2^6,
\end{multline*}
\normalsize
which implies that
\scriptsize
\begin{multline*}
m_{00}=\frac{7}{2}, m_{10}=7, m_{01}=\frac{28}{3}, m_{20}=\frac{175}{12}, m_{11}=\frac{455}{24}, m_{02}=\frac{329}{12}, m_{30}=\frac{63}{2}, 
m_{21}=\frac{2387}{60},\\ 
m_{12}=\frac{3591}{20}, m_{03}=\frac{175}{2}, m_{40}=\frac{2107}{30}, m_{31}=\frac{5159}{60}, m_{22}=\frac{21217}{180}, m_{13}=\frac{3591}{20},
m_{04}=\frac{2989}{10}, \\ 
m_{50}={161}, m_{41}=\frac{2863}{15}, m_{32}=\frac{7609}{30}, m_{23}=\frac{5642}{15}, m_{14}=\frac{3066}{5}, m_{05}=1078,  m_{60}=\frac{3025}{8},\\
m_{51}=\frac{6945}{16}, m_{42}=\frac{13397}{24}, m_{33}=\frac{128589}{160}, m_{24}=\frac{153493}{120}, m_{15}=\frac{35333}{16}, m_{06}=\frac{32467}{8}.
\end{multline*}
\endgroup
\end{example}

\subsection*{Results on non-convex polytopes.}
Our second group of results addresses the problem of distinguishing different polytopes with the same underlying set of vertices from information on their moments. The problem of restoring  the vertices  of a polygon  or a polytope with a constant mass density from information on its moments was addressed earlier in e.g.
\citelist{\cite{Bro1} \cite{BroSt} \cite{GMP} \cite{GGMPV} \cite{GMV00} \cite{GLPR}}. However, 
the latter do not provide the recovery of the vertices in the generality required in the present paper.  
Below we concentrate on the case of constant density and known vertices, and plan to return to the general inverse problem for polytopes with unknown polynomial  density and unknown  location of their vertices in the future. 

First we need to define what we mean by a polytope.   It turned out  that there is no general consensus about this notion. Instead there exist several competing definitions having their own advantages in different situations.  %We assume that  our readers are familiar with the standard notion of a simplicial complex and its faces and
  We shall study the following class of polytopal objects.

\newcommand{\gp}{generalized polytope}
\begin{defi+} 
A subset $\cP\subset \bR^d$  coinciding with a finite union of arbitrary convex  
$d$-dimensional polytopes is called  a {\it \gp}. 
\end{defi+}

\begin{defi+}
The {\it number of
components} of a   generalized polytope $\cP$ is the number of connected
components of the set $\cP^o\subset \cP$ of interior points of  $\cP$. The
closure of each connected component of $\cP^o$ is called a {\it component}  of
$\cP$. A generalized polytope with one component is called {\it
indecomposable}. 
\end{defi+}

\begin{rema+} We say that a simplicial complex in $\bR^d$ is {\it pure} if all its maximal simplices have dimension $d$.  
Clearly any generalized polytope  in $\bR^d$ can be represented as the topological space of an appropriate  pure simplicial
complex. 
\end{rema+}

\begin{rema+}
Often one considers a more restricted class of objects, namely
{\it polytopes}. A polytope $\cP\subset \bR^d$ is a \gp\ 
homeomorphic to a $d$-dimensional manifold with boundary.  
\end{rema+}
%Each traingulation $\cT$ of a \gp\ $\cP\subset\bR^d$ (note that we allow arbitrarily fine finite subdivisions
%of $\cP$) specifies a collection $F_{\cT}(\cP)$ of flats 
%(i.e. affine subspaces of $\bR^d$), spanned by faces of each simplex in $\cT$. A flat will be called 
%$\cP$-{\em stable} if it appears in $F_{\cT}(\cP)$ for {\em any} triangulation $\cT$.
%\begin{defi+} 
%A {\em facet} (or a $d-1$-dimensional {\em face}) 
%of $\cP$ is the closure of a $d-1$-dimensional component of the  interior of $S\cap\partial\cP$,
%for $S$ a $\cP$-stable $d-1$-dimensional flat.\\
%More generally, a $k$-dimensional {\em face} of $\cP$, for $k\leq d-2$, is the closure of a $k$-dimensional 
%component of the  interior of $S\cap\bigcup_{X\subset\cP}\partial X$,
%where the union is taken over all $k+1$-dimensional flats $X$, and $S$ is a $\cP$-stable $k$-dimensional flat. \\
%A 0-dimensional flat of $\cP$ is called a {\em vertex}.
%\end{defi+}
%\begin{proposition}\label{prop:strat} a) Each \gp\ has a unique stratification into faces where each face is an indecomposable \gp. 

%\noindent
%b) For each face $f\subseteq \cP$ the affine subspace spanned by $f$ coincides with the affine subspace spanned by its set of vertices. 

%\noindent
%c) A vertex of $\cP$ is a 0-dimensional $\cP$-stable flat. 
%\qed
%\end{proposition}

We need to introduce the notion of a vertex of a generalized polytope.  

\begin{defi+} Given a generalized polytope $\cP\subset \bR^d,$ we call a finite collection of open  disjoint $d$-dimensional simplices in $\bR^d$ a {\it dissection} of $\cP$ if  the closure of their union coincides with $\cP$.  
\end{defi+}
A wealth of material on dissections of polytopes can be found in \cite{Pak08}, see also
\cite{MR2743368}.
\begin{defi+}
Given a generalized polytope $\cP\subset \bR^d$, we call a point $\bv$  a {\it vertex} of $\cP$, 
if $\bv$ is a vertex of (the closure of) some open simplex in every dissection of $\cP$. 
\end{defi+} 

\begin{defi+} Given a point $p\in \cP$ of a generalized polytope $\cP,$ we denote by the tangent cone $T_p(\cP)$ of $\cP$ at $p$ the set obtained as follows. For a sufficiently small $\eps>0,$ set $\cP_p(\eps)=\cP\cap B_p(\eps)$ where $B_p(\eps)$ is the $\eps$-ball centered at $p$. Define $T_p(\cP)$ as the set obtained by taking a ray through $p$ and every point of $\cP_p(\eps)$. In other words,  $T_p(\cP)$ is the cone with the apex at $p$ and the base $B_p(\eps)$. (Obviously, 
$T_p(\cP)$ is independent of $\eps$ for a sufficiently small $\eps>0$, and it need not be convex.) 
\end{defi+}

\begin{lemma}\label{lm:inv}
A point $\bv$ is a {\it vertex} of $\cP$ if and only if $T_\bv(\cP)$ does not admit a decomposition in 
the disjoint union of convex polyhedral subcones, such that each subcone in the decomposition has a translation-invariant direction (i.e. is not pointed). 
In particular, if the tangent cone to $\cP$ at $\bv$ has a connected component with no translation-invariant direction, then $\bv$ is a vertex. 
\end{lemma}  

We denote by $\conv(S)$ the convex hull of an arbitrary set $S\subset\bR^d$.
The above lemma implies that any vertex of $\conv(\cP)$ is a vertex of $\cP$.

The following result extends Corollary~\ref{cor:arbit} to the case of generalized polytopes. 

\begin{prop+}\label{prop:gener}
For any generalized polytope $\cP$  with the set of  vertices $\Verti(\cP)$,
the denominator of its normalized moment generating function $F^\rho_\cP(\bu)$ 
with respect to  a homogeneous
polynomial density function $\rho$ of degree $\delta$ divides 
$$\Phi_\cP(\bu):=\prod_{\bv\in \Verti(\cP)}(1-\lan \bv,\bu\ran)^\delta.$$
\end{prop+}

\begin{rema+} \label{rema:schon}
There exist generalized polytopes which do not admit
dissections with only existing vertices.  The simplest example of this kind 
is the {\em Sch\"onhardt polyhedron}, see Figure~\ref{fig:schon} and  \cite{Sch}. Absence of 
a dissection $\cT$ which uses only its $6$ vertices can be established by observing that none of
the edges $AC$, $A'B$, and $B'C'$ can appear in a simplex of $\cT$, yet 
any simplex on these $6$ vertices must contain one of them. Therefore,
Proposition~\ref{prop:gener} is not an immediate consequence of
Corollary~\ref{cor:simplex}.  
\end{rema+}

\begin{figure}[h]
\begin{tikzpicture}[z=-5.5,line join=bevel,scale=2]%,z=-5.5]
\coordinate (Ap) at (-0.07,-0.05,-1); % A'
\coordinate (Cp) at (-1,0.05,-0.15);  % C'
\coordinate (A) at (-0.07,-0.05,1);  % A
\coordinate (C) at (1,0.05,-0.15);   % C
\coordinate (B) at (0.05,1,0.25);    % B
\coordinate (Bp) at (0.05,-1,0.25);   % B'

%\coordinate (Ap) at (-0.07,-0.05,-1); % A'
%\coordinate (Cp) at (-1,0.05,-0.05);  % C'
%\coordinate (A) at (-0.07,-0.05,1);  % A
%\coordinate (C) at (1,0.05,-0.05);   % C
%\coordinate (B) at (0.05,1,0.05);    % B
%\coordinate (Bp) at (0.05,-1,0.05);   % B'

\coordinate (S1) at (-0.5,0.3,0);
\coorshifting
\draw [fill opacity=0.7,fill=green!70!blue] (CC) -- (AAp) -- (BB) -- cycle; 
\draw (CC) -- (AAp) -- (CCp) -- cycle; 
\draw (CC) -- (BB) -- (CCp) -- cycle; 
\draw [fill opacity=0.7,fill=green!70!blue]  (AAp) -- (BB) -- (CCp) -- cycle; 

\coordinate (S1) at (-0.6,-0.25,0);
\coorshifting
\draw (AA) -- (AAp) -- (BBp) -- cycle; 
\draw (AA) -- (AAp) -- (CCp) -- cycle; 
\draw [fill opacity=0.7,fill=gray!70!black]   (AA) -- (BBp) -- (CCp) -- cycle; 
\node (z) at (0,0,0) [label=left:$0$]{};
\fill [blue] ($(z)$) circle (0.3pt);
\draw (Cp) node (a2) [label=left:$C'$]{};

\draw (C) -- (Cp) -- (Ap) -- cycle;
\draw (C) -- (Cp) -- (B) -- cycle;
\draw (Ap) -- (A) -- (Cp) -- cycle;

\draw (C) -- (Ap) -- (Bp) -- cycle;
\draw [fill opacity=0.7,fill=green!80!blue] (Cp) -- (A) -- (B) -- cycle;
\draw (Ap) -- (A) -- (Bp) -- cycle;
\draw [fill opacity=0.8,fill=orange!80!black] (A) -- (Bp) -- (B) -- cycle;
\draw [fill opacity=0.7,fill=purple!70!black] (C) -- (B) -- (Bp) -- cycle;

\draw (A) node (a3) [label=-120:$A$]{};
\draw (C) node (a4) [label=right:$C$]{};
\draw (B) node (b1) [label=90:$B$]{};
\draw (Bp) node (c1) [label=-90:$B'$]{};
\draw (Ap) node (a1) [label=45:$A'$]{};

\coordinate (S1) at (0.53,0,0);
\coorshifting
\draw (AA) -- (BB) -- (BBp) -- cycle; 
\draw [fill opacity=0.7,fill=green!80!blue] (AA) -- (BB) -- (CC) -- cycle; 
\draw [fill opacity=0.7,fill=gray!70!black]  (AA) -- (BBp) -- (CC) -- cycle; 

% now we draw the initial tetrahedron:
\coordinate (S1) at (-3.2,0,0);
\coorshifting
\draw (CC) node (d2) [label=-90:$C$]{};
\draw (CCp) node (c2) [label=-90:$C'$]{};
\draw (AAp) node (a1) [label=45:$A'$]{};
\draw (BB) node (b1) [label=90:$B$]{};
\draw (BBp) node (c1) [label=-90:$B'$]{};

\draw[ultra thin] (AAp) -- (AA);
\draw[semithick] (CCp) -- (CC);
\draw[very thick] (BBp) -- (BB);

%\draw (AAp) -- (AA);
%\draw (BBp) -- (BB);
%\draw (CCp) -- (CC);
\draw (AAp) -- (BB) -- (CCp) -- cycle; 
\draw (AAp) -- (BBp) -- (CCp) -- cycle; 
\draw (AAp) -- (BB) -- (CC) -- cycle; 
\draw (AAp) -- (BBp) -- (CC) -- cycle; 

\draw [fill opacity=0.7,fill=green!80!blue] (AA) -- (BB) -- (CC) -- cycle; 
\draw [fill opacity=0.7,fill=green!80!blue] (AA) -- (BB) -- (CCp) -- cycle; 
\draw [fill opacity=0.7,fill=gray!70!black]  (AA) -- (BBp) -- (CC) -- cycle; 
\draw [fill opacity=0.7,fill=gray!70!black]  (AA) -- (BBp) -- (CCp) -- cycle; 

\draw (AA) node (a3) [label=-120:$A$]{};
\end{tikzpicture}
\caption{Sch\"{o}nhardt polyhedron  
obtained from an octahedron (on the left) by removing  tetrahedra 
$[ABB'C]$, $[AA'B'C']$, and $[A'BCC']$.
\label{fig:schon}}
\end{figure}
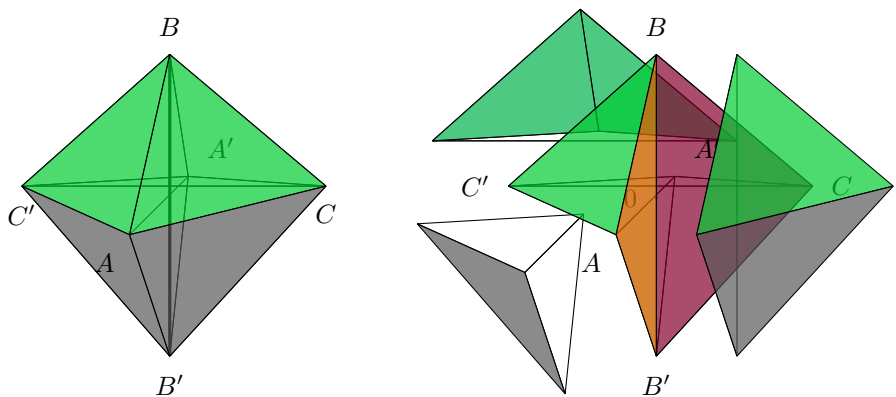

\begin{rema+}
For ``generic'' generalized polytopes $\cP$, the denominator $\Omega(\bu)$ of $F_\cP(\bu)$ 
equals  $\Phi_\cP(\bu)$, but for certain special polytopes  the denominator $\Omega(\bu)$ 
may be its proper divisor, as can be seen from 
the following example. Let $A=\{0,a_1,a_2,a_3\}\subset\RR^3$ be a spanning
set, and $v\in\RR^3$.
Let $\cP_\pm:=\conv(v\pm A)$ and
$\cP:=\cP_+\cup\cP_-$.
Then $1-\lan\bu,v\ran$ does not appear in $\Omega(\bu)$, as
$$F_\cP(\bu)=F_{\cP_+}(\bu)+F_{\cP_-}(\bu)=K\frac{\sum\limits_{1\leq i<j\leq 3}\lan
\bu,a_i\ran\lan\bu,a_j\ran +(1-\lan\bu,v\ran)^2}
{\prod\limits_{1\leq i\leq 3}((1-\lan\bu, v\ran)^2-\lan \bu,a_i\ran^2)},$$
where $K\neq 0$ is a real constant. 
\end{rema+}

%\begin{not+}
Now we  introduce several finite-dimensional linear spaces related to a given finite spanning set $S\subset \bR^d$. Let $\cP(S)$ be the set of all generalized polytopes $\cP$  whose sets $\Verti(\cP)$ of vertices are contained in   $S$. %(Part b) of Proposition~\ref{prop:strat} implies that  $\cP(S)$ is finite.) 
For $\cP\in \cP(S)$, we denote by $\mu_\cP$ its standard measure.  (Obviously, $\mu_\cP$ is supported on $\cP\subseteq \conv(S)$.) 

Denote by $\M(S)$ the linear space of all signed measures, i.e. the linear span of all standard measures $\mu_\cP$ for $\cP\in \cP(S)$. Let $\M^\Delta(S)\subseteq \M(S)$ be its subspace spanned by  $\mu_\Delta$, for $\Delta\in \cP(S)$ a $d$-dimensional simplex.  (The space $\M^\De(S)$ has earlier appeared in \cite{AlGelZel}, \cite{Al1}, \cite{Al2} in a somewhat different context.) We shall refer to  elements of $\M(S)$ as to {\it polytopal measures} with the vertex set $S$. The following conjecture was central to our study; it was
pointed out to us that it follows from results in \cite{BrVe99}
after \cite{GPSS12} was released (cf. Subsection~\ref{subs:arr}
for a discussion). As well, at the same time authors of \cite{ABR17}
started working on this question; their
 \cite{ABR17}*{Theorem~1}, proved using a distinct from \cite{BrVe99}
 set of ideas, implies the conjecture.
\begin{conjecture}\label{conj:main} 
(Corollary to \cite{ABR17}*{Theorem~1}) For an arbitrary spanning set $S$
and any $\cP\in\cP(S)$, its standard measure $\mu_\cP$ belongs to $\M^\Delta(S)$. In other words, $\M(S)=\M^\Delta(S)$.  %Moreover, for any spanning set $S$ the  dimension of the space $\M(S)=\M^\Delta(S)$ depends only on the (non-oriented) matroid determined by $S$.
\end{conjecture}

By Remark~\ref{rema:schon}, the above is non-trivial.  In fact, \cite{ABR17}
shows a stronger result, namely that the coefficients in a decomposition of
$\mu_cP$ into a sum of $\mu_\Delta$ are  integers, in particular
resolving in the affirmative \cite{GPSS12}*{Problem~3}.
In view of this, we can
make a stronger, ``inclusion-exclusion``-like conjecture.

\begin{conjecture}\label{conj:inclexcl}
For an arbitrary spanning set $S$ and
any $\cP\in\cP(S)$, its
standard measure $\mu_\cP$ can be decomposed as
\[
\mu_\cP=\sum_{\Delta\in\mathcal{D}_\cP} \sigma_\Delta\mu_\Delta,
\qquad \sigma_\Delta=\pm 1\text{ for all $\Delta$},
\]
with $\mathcal{D}_\cP$ a set of $d$-dimensional simplices in $\cP(S)$.
\end{conjecture}
Note that this holds true for $d=2$, as well as for any convex
$\cP$, with a stronger condition that all $\sigma_\Delta=1$.

While we did not have a proof of Conjecture~\ref{conj:main} in its full generality, we have
succeeded in proving it for a rather large class of spanning sets. 
Roughly speaking, the latter should be close to ``generic''.
Specifically, given a finite spanning set $S\subset \bR^d,$  we say that $S$ is {\it weakly non-degenerate} if any $(d+2)$-tuple of points from $S$  is spanning.  If $S$ satisfies the stronger condition that each $(d+1)$-subset of $S$ is spanning then we call  the latter $S$  {\em strongly non-degenerate}.

\begin{theorem}\label{th:fund} 
Conjecture~\ref{conj:main} holds for any weakly non-degenerate finite set $S$. 
\end{theorem} 

\begin{rema+} Theorem~\ref{th:fund} would imply Conjecture~\ref{conj:main} if one could prove that the standard measure of an arbitrary generalized polytope $\cP$ can be obtained as the limit of the  standard measures  of a $1$-parameter family of generalized polytopes $\cP(t)$ with $\cP(0)=\cP$  such that for $t\neq 0$ the vertices of $\cP(t)$ are weakly non-degenerate, and 
each vertex of $\cP(t)$ tending to a vertex of $\cP$
as $t\to 0$. We are unable to prove the existence of such deformations in general. 
\end{rema+} 

The key idea in the proof of  Theorem~\ref{th:fund} is to study the 
corresponding spaces of Fantappi\`{e}  transformations of signed measures in $\M(S)$. 
In particular, we are able to compute the corresponding dimensions\footnote{
Note that presently we are not aware of a formula or a recipe for calculating the dimension of 
$\M^\De(S)$ without the assumption of the theorem. 
(The manuscript \cite{Al1} contains an algorithm constructing a basis of this space.) }.
In more detail, 
let $\F(S)$ (resp. $\F^\Delta(S)$) be the linear space of Fantappi\`{e}  transformations of signed measures in $\M(S)$ (resp. $\M^\Delta(S)$). In other words, $\F(S)$  (resp. $\F^\Delta(S)$) is the space of normalized moment generating functions of signed  measures in $\M(S)$ (resp. $\M^\Delta(S)$). %Obviously, $\mathfrak F^\Delta_S\subseteq \mathfrak F_S.$

Since each compactly supported measure is uniquely determined by its complete set of moments, the map 
\begin{equation}\label{eq:isom}
F_\mu: \M(S)\to \F(S),
\end{equation}
induced by the Fantappi\`{e} transformation is a linear isomorphism, cf. \cite{APS}*{Section 3.5}.

%Given a finite spanning set $S\subset \bR^d$  and its arbitrary point $\bv\in S$ associate to the pair $(S,\bv)$ the vector configuration $Vect(S,\bv)$ consisting of all vectors $\bw-\bv$ where $\bw\in S$. Since $S$ is spanning then  $Vect(S,\bv)$ spans $\bR^d$. Denote by $Matr(S,\bv)$ the associated vectorial matroid and by $\sharp(S,\bv)$ the number of bases in $Matr(S,\bv)$, see e.g. \cite{BVSWZ}. 

%\end{not+}

Finally, given a spanning set $S=\{\bv_1,\ldots , \bv_N\}\subset \bR^d,$ denote by $\Rat(S)$ the linear space of all rational functions with the denominator $\Phi_S(\bu)$ as in \eqref{eq:denom},
\begin{equation}\label{eq:denom}
\Phi_S(\bu)=\prod_{i=1}^N (1-\lan\bv_i,\bu\ran),
\end{equation}
and with the numerator an arbitrary real (inhomogeneous) polynomial of degree at most $N-d-1$. Here the 
numerator and the denominator might have common factors.

\begin{proposition}\label{pr:generic} 
  $\F^\Delta(S)$  coincides with $\Rat(S)$ if and only if 
$S$ is strongly non-degenerate. 
%\item If $S$ is strongly non-degenerate then for any $i=1,\ldots, N$ 
%the set $\B_{i}$ of the standard measures  of all simplices containing $\bv_i$ 
%is a natural basis of $\F^\Delta(S)$.  
%\item If $S$ is degenerate then  $\dim\Rat(S)<\binom{N-1}{d}$. \label{prop:degcase}
%\end{enumerate}
\end{proposition}

\begin{corollary}\label{cor:generic} If $S$ is strongly non-degenerate then  $\M^\Delta(S)=\M(S)$. \end{corollary}

Corollary~\ref{cor:generic} implies that for  strongly non-degenerate $S$, the
dimension  of all these   linear spaces equals $\binom{N-1}{d}$. Note that
Corollary~\ref{cor:generic} settles Theorem~\ref{th:fund} for the strongly
non-degenerate $S$. 

\medskip
Our final  goal is  to explicitly solve  the following inverse moment problem. 
\medskip

\begin{problem}
 Given a strongly non-degenerate spanning set $S\subset \bR^d$, $|S|=N$, find the unique 
 polytopal measure in  $\M(S)$ with a given  set of all moments up to order $N-d-1$. 
\end{problem}
  
  \medskip
  We start with the following simple observation.

\begin{lemma}\label{lm:McL} Given an arbitrary spanning set $S\subset \bR^d$, $|S|=N$, and an arbitrary polynomial $T(\bu)$ of degree at most $N-d-1$, there exists a unique rational function $R(\bu)=P(\bu)/\Phi_S(\bu)$ with Taylor polynomial of degree $N-d-1$ at the origin equal to 
$T(\bu)$. Namely, $P(\bu)=\left[T(\bu) \Phi_S(\bu)\right]_{N-d-1}$, where $\left[\cdot\right]_{N-d-1}$ 
stands for the truncated polynomial with all monomials up to degree $N-d-1$. %(By the Maclaurin polynomial we mean the Taylor polynomial at the origin.) 
\end{lemma}

For $S=\{\bv_1,\ldots , \bv_N\}\subset \bR^d$ strongly non-degenerate, we give an explicit inversion formula determining the densities of an unknown polytopal measure  having a given set of moments up to order $N-d-1$  on each  simplex in a natural basis of $\M^\Delta(S)$. In view of Lemma~\ref{lm:McL} we can assume that we are already given an arbitrary rational function $R(\bu)=P(\bu)/\Phi_S(\bu)$, where $\deg P(\bu)\le N-d-1$, and we want to determine the densities of the required signed measure from $\M(S)$ in terms of  numerator $P(\bu)$. 

From now on we shall choose the basis of $\M^\Delta(S)$ consisting of the standard measures of all simplices containing the last vertex $\bv_N$, see Lemma~\ref{lm:simplex} below. Let  $\LL=\{l_1,l_2,....,l_{N-1}\}$ be the $(N-1)$-tuple of  linear forms corresponding to  vertices $\bv_1,\bv_2,\ldots, \bv_{N-1}$, where $l_i(\bu)=1-\lan\bv_i,\bu\ran$.  Consider the linear span $V_\LL$ of all possible products of the form $l_{j_1}\cdot l_{j_2} \cdot ... \cdot l_{j_{N-d-1}},\; 1\le j_1<j_2<...<j_{N-d-1}$.  There are $\binom{N-1}{d}$ such products, and each of them is a polynomial of degree at most  $N-d-1$.  On the other hand,  the dimension of the space $Pol(N-d-1,d)$ of all (inhomogeneous) polynomials of degree at most $N-d-1$ in $d$ variables equals $\binom{N-1}{d}$, as well.  

Define the square matrix $Mat_{S}$ of size $\binom{N-1}{d}$ with entries being coefficients of the above products of linear forms with respect to  the standard monomial basis in $Pol(N-d-1,d)$. We assume that $Mat_{S}$ acts on the space $V_\LL$ of  {\em column} vectors.

\begin{theorem}\label{th:inverse}
For an arbitrary strongly non-degenerate spanning set $S\subset \bR^d,$ $|S|=N$, the matrix $Mat_{S}$ is invertible. Moreover, for a   rational function $R(\bu)=P(\bu)/\Phi_S(\bu)$, where  $P(\bu)$ is an arbitrary polynomial of degree $N-d-1$, there exists a unique  measure $\mu_R\in \M(S)$ with Fantappi{\`e} transform $R(\bu)$. Namely, 
\begin{equation}\label{eq:inv}
\mu_R=Mat_{S}^{-1}(P(\bu)).
\end{equation}
\end{theorem}

\medskip
\begin{rema+}
A detailed explanation  of the meaning of \eqref{eq:inv}  can be found in the proof of  Theorem~\ref{th:inverse}, see also Example~\ref{ex:two} below. 
An explicit formula for the matrix $Mat_{S}^{-1}$  is given in Lemma~\ref{lm:mat}. 
\end{rema+}
\medskip

Recall that  a spanning set $S$ is weakly non-degenerate if any $(d+2)$-tuple
of its points is spanning. With minor changes, the above solution of the inverse
moment problem can be adapted to this more general
case. In order not to overload the introduction we refer the readers interested
in this situation to Section~\ref{s4}.  The case of an arbitrary spanning set
$S$, however, remains unsolved and offers several interesting challenges in
matroid theory. We hope to return to it in the future. 

It will be convenient to work with scaled volumes of simplices, which we call {\em weights}.
\begin{defi+}  Given a signed measure $\mu$ in $\bR^d$ and a $d$-dimensional simplex $\De\subset \bR^d$, we define the {\it weight} $w_\De$ of $\De$ by the formula: 
\begin{equation}\label{eq:weight}
w_\De=d!\int_\De d\mu.
\end{equation}
 In other words, the density $d_\De$ of the measure in question  which should be placed at $\De$ equals $$d_\De=\frac{w_\De}{d! \Vol(\De)}.$$
\end{defi+} 

We finish the introduction by explicitly solving  the above inverse problem for a concrete $5$-tuple of points in $\bR^2$.
 
\begin{example}\label{ex:two} 
Set $S=\{\bv_1,\bv_2,\bv_3,\bv_4,\bv_5\}$ where $\bv_1=(1,0), \bv_2=(2,1), \bv_3=(1,2), \bv_4=(0,1), \bv_5=(0,0)$.  The corresponding set $\LL=\{l_1,l_2,l_3,l_4\}$ of linear forms is given by $l_1=1-u_1,\,   l_2=1-2u_1-u_2,\,   l_3=1-u_1-2u_2,\,   l_4=1-u_2$. Additionally, $l_5=1$. We are considering the basis of $\M^\De(S)$ consisting of (the standard measures of) 6 triangles containing $\bv_5$. Therefore we need 6 quadratic forms obtained as pairwise products $l_il_j, \, 1\le i<j \le 4$. We get 
$$\begin{cases}
l_1l_2=1-3u_1-u_2+2u_1^2+u_1u_2\\
l_1l_3=1-2u_1-2u_2+u_1^2+2u_1u_2\\
l_1l_4=1-u_1-u_2+u_1u_2\\
l_2l_3=1-3u_1-3u_2+2u_1^2+5u_1u_2+2u_2\\
l_2l_4=1-2u_1-2u_2+2u_1u_2+u_2^2\\
l_3l_4=1-u_1-3u_2+u_1u_2+2u_2^2.\\
\end{cases}
$$
Notice that $l_1l_2$ corresponds to  triangle $\De_{345}$, $l_1l_3$ to $\De_{245}$, $l_1l_3$ to $\De_{245}$, $l_1l_4$ to $\De_{234}$, $l_2l_3$ to $\De_{145}$, $l_2l_4$ to $\De_{135}$, and $l_3l_4$ to $\De_{125}$.  
Ordering monomials spanning the space $Pol(2,2)$ as $(1,u_1,u_2,u_1^2,u_1u_2, u_2^2)$, we get the $6\times 6$-matrix $Mat_S$ and its inverse $Mat_S^{-1}$ as follows

\[
\begin{blockarray}{rrrrrrr}
\begin{block}{cc\BAmulticolumn{4}{c}c}
&&Mat_S= \\
\end{block}
 & l_1l_2 & l_1l_3 & l_1l_4 & l_2l_3&l_2l_4&l_3l_4 \\
\begin{block}{r(rrrrrr)}
1&  1& 1& 1& 1& 1& 1 \\
u_1   &-3&-2&-1&-3&-2&-1\\
u_2   &-1&-2&-1&-3&-2&-3\\
u_1^2 &2&1&0&2&0&0\\
u_1u_2&1&2&1&5&2&1\\
u_2^2 &0&0&0&2&1&2\\
\end{block}
\end{blockarray}
\quad  
\begin{blockarray}{rrrrrrr}
\begin{block}{cc\BAmulticolumn{4}{c}c}
&& 4Mat_S^{-1}= \\
\end{block}
 & 1 & u_1 & u_2 & u_1^2&u_1u_2&u_2^2 \\
\begin{block}{r(rrrrrr)}
l_1l_2&1&-1&1&1&1&-1\\
l_1l_3&-4&0&-4&0&0&-4\\
l_1l_4&9&3&3&1&1&1\\
l_2l_3&1&1&1&1&1&1\\
l_2l_4&-4&-4&0&-4&0&0\\
l_3l_4&1&1&-1&1&-1&1\\
\end{block}
\end{blockarray}
\]
\iffalse
\[
\quad Mat_S=\qquad\qquad\qquad  \qquad   \qquad     4Mat_S^{-1}= 
\]
\[
\bordermatrix{~ & l_1l_2 & l_1l_3 & l_1l_4 & l_2l_3&l_2l_4&l_3l_4 \cr 
               1&       1&       1&       1&      1&     1&     1 \cr
u_1   &-3&-2&-1&-3&-2&-1\cr
u_2   &-1&-2&-1&-3&-2&-3\cr
u_1^2 &2&1&0&2&0&0\cr
u_1u_2&1&2&1&5&2&1\cr
u_2^2 &0&0&0&2&1&2}   
\quad
\bordermatrix{~ & 1 & u_1 & u_2 & u_1^2&u_1u_2&u_2^2 \cr  
l_1l_2&1&-1&1&1&1&-1\cr
l_1l_3&-4&0&-4&0&0&-4\cr
l_1l_4&9&3&3&1&1&1\cr
l_2l_3&1&1&1&1&1&1\cr
l_2l_4&-4&-4&0&-4&0&0\cr
l_3l_4&1&1&-1&1&-1&1}. 
\]
%\iffalse
$$Mat_S=\bordermatrix{~ & l_1l_2 & l_1l_3 & l_1l_4 & l_2l_3&l_2l_4&l_3l_4 \cr 
1&1&1&1&1&1&1\cr
u_1&-3&-2&-1&-3&-2&-1\cr
u_2&-1&-2&-1&-3&-2&-3\cr
u_1^2&2&1&0&2&0&0\cr
u_1u_2&1&2&1&5&2&1\cr
u_2^2&0&0&0&2&1&2}
\quad  4Mat_S^{-1}=\bordermatrix{~ & 1 & u_1 & u_2 & u_1^2&u_1u_2&u_2^2 \cr  
l_1l_2&1&-1&1&1&1&-1\cr
l_1l_3&-4&0&-4&0&0&-4\cr
l_1l_4&9&3&3&1&1&1\cr
l_2l_3&1&1&1&1&1&1\cr
l_2l_4&-4&-4&0&-4&0&0\cr
l_3l_4&1&1&-1&1&-1&1}. 
$$
\fi
(For \TeX nical reasons we give $4Mat_S^{-1}$ above.) 
Thus, given an arbitrary  rational function $R(u_1,u_2)=P(u_1,u_2)/\Phi_S(u_1,u_2)$ where  
$P(u_1,u_2)=a_{00}+a_{1,0}u_1+a_{0,1}u_2+a_{2,0}u_1^2+a_{11}u_1u_2+a_{02}u_2^2$ is a
polynomial  of degree at most $2$ and $\Phi_S(u_1,u_2)=l_1l_2l_3l_4l_5$, we get 
$$\begin{cases}
w_{345}=\frac{1}{4}(a_{00}-a_{10}+a_{01}+a_{20}-a_{11}+a_{02})\\
w_{245}=-a_{00}-a_{01}-a_{02}\\
w_{235}=\frac{1}{4}(9a_{00}+3a_{01}+3a_{10}+a_{20}+a_{11}+a_{02})\\
w_{145}=\frac{1}{4}(a_{00}+a_{01}+a_{10}+a_{20}+a_{11}+a_{02})\\
w_{135}=-a_{00}-a_{10}-a_{20}\\
w_{125}=\frac{1}{4}(a_{00}+a_{10}-a_{01}+a_{20}-a_{11}+a_{02}),\\
\end{cases}
$$
where $w_{ijk}$ is the weight of the signed measure to be placed on $\De_{ijk}$,  see \eqref{eq:weight}.

To illustrate  all steps of  solution of our inverse moment problem, %with the fixed set of vertices 
assume that we are looking for a polygonal measure with the vertex set $S$
and (ad hoc chosen) moments $m_{00}=1, m_{10}=2, m_{01}=3, m_{20}=4, m_{11}=5, m_{02}=6$. Then its normalized moment generating function $F_\mu(\bu)$ satisfies the relation 
$$F_\mu(\bu)= 1\frac{2!}{0!0!}+2\frac{3!}{1!0!}u_1+3\frac{3!}{0!1!}u_2+4\frac{4!}{2!0!}u_1^2+5\frac{4!}{1!1!}u_1u_2+6\frac{4!}{0!2!}u_2^2+\dots=\frac{P(u_1,u_2)}{l_1l_2l_3l_4l_5},$$
where $P(u_1,u_2)$ is a (non-homogeneous) polynomial of at most second degree. Thus, truncating the product of the left-hand side and ${l_1l_2l_3l_4l_5}$ up to the second degree, we obtain
$$P(u_1,u_2)= 2+4u_1+10u_2+10u_1^2+24u_1u_2+10u_2^2,$$
i.e. $a_{00}=2, a_{10}=4, a_{01}=10, a_{20}=10, a_{11}=24, a_{02}=10$. Thus  $w_{345}=1,  w_{245}=-22,  w_{235}=26, 
w_{145}=15, w_{135}=-16,  w_{125}=-2$. The areas of the corresponding triangles are equal to:  
$Area{(\De_{345})}=\frac{1}{2}; Area({\De_{245}})=1;  Area{(\De_{235})}=\frac{3}{2}; Area{(\De_{145})}=\frac{1}{2}; Area{(\De_{135})}=1; Area{(\De_{125})}=\frac{1}{2}.$  This implies that the densities of the measure of the corresponding triangles are equal to $d_{345}=1, d_{245}=-11, d_{235}=\frac{26}{3},  d_{145}=15, d_{135}=-8,    d_{125}=-2$.  To obtain the final  densities in the convex hull $\conv(S)$ of $S$, one has to decompose $\conv(S)$ into domains obtained by removing from $\conv(S)$  the set of all hyperplanes  spanned by  vertices in $S$. For each such domain,  we should add up the densities of all basic simplices containing this domain.  The resulting measure  is shown in Fig.~\ref{fg:T1}. 

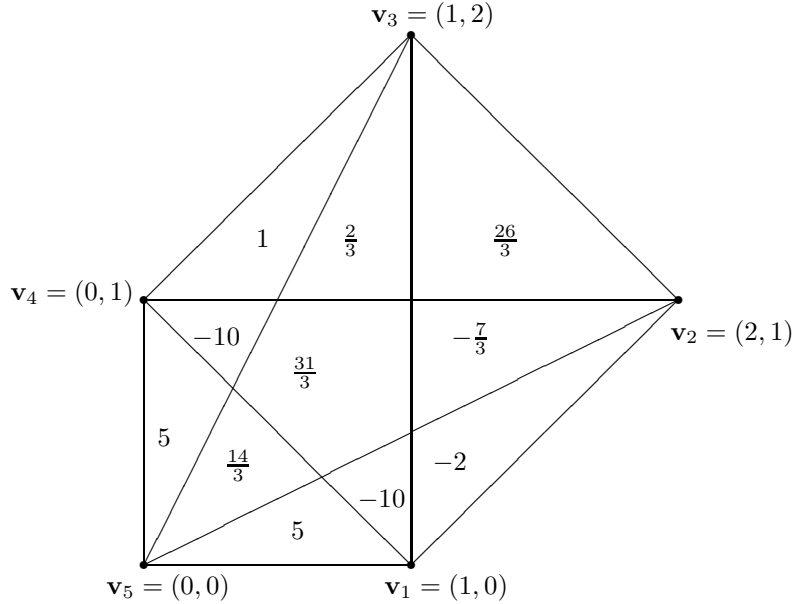
\begin{figure}[h]
\begin{center}
\begin{picture}(360,230)(0,0)
\put(120,20){\line(1,0){100}}
\put(220,20){\line(1,1){100}}
\put(120,20){\line(0,1){100}}
\put(120,120){\line(1,1){100}}
\put(220,220){\line(1,-1){100}}
\put(120,20){\line(2,1){200}}
\put(120,20){\line(1,2){100}}
\put(120,120){\line(1,-1){100}}
\put(120,120){\line(5,0){200}}
\put(220,20){\line(0,5){200}}

\put(106,9){$\bv_5=(0,0)$}
\put(210,9){$\bv_1=(1,0)$}
\put(317,105){$\bv_2=(2,1)$}
%\put(325,120){$\bv_2=(2,1)$}
\put(205,225){$\bv_3=(1,2)$}
\put(70,120){$\bv_4=(0,1)$}

\put(175,30){$5$}
\put(200,42){$-10$}
\put(228,56){$-2$}
\put(250,140){$\frac{26}{3}$}
\put(175,90){$\frac{31}{3}$}
\put(235,103){$-\frac{7}{3}$}
\put(194,140){$\frac{2}{3}$}
\put(162,140){$1$}
%\put(210,105){$3$}
\put(150,55){$\frac{14}{3}$}
\put(125,65){$5$}
%\put(140,125){$1$}
\put(138,103){$-10$}

\put(120,20){\circle*{3}}
%\put(120,20){\circle{8}}
\put(220,20){\circle*{3}}
\put(320,120){\circle*{3}}
\put(220,220){\circle*{3}}
\put(120,120){\circle*{3}}

\end{picture}
\end{center}
\caption{Final measure in Example~\ref{ex:two}.  \label{fg:T1}}
\end{figure}
\end{example}

\begin{rema+} Domains into which the convex hull $\conv(S)$ is cut by the hyperplanes spanned by $S$ were introduced in \cite{AlGelZel} where they were called {\it chambers}.  The incidence matrix of the simplices spanned by $S$ and those chambers was studied in some detail in \cite {Al1},\cite{Al2}. This matrix allows to formalize the last step of construction of the above polygonal measure, where information on the densities of the simplices is transformed into information on the densities of the chambers.  But, in general, already the number of chambers is a complicated invariant of the set $S$. It seems that the general problem of constructing the set of chambers and the corresponding incidence matrix in terms of  a given $S$  is quite non-trivial.  
\end{rema+} 

\medskip
\begin{ack}
The second author is grateful to the Mathematics Department of Stockholm
University for the hospitality in June 2011 when this project was initiated.
The third author wants  to acknowledge the hospitality of the School of
Physical and Mathematical Sciences, Nanyang Technological University  in April
2012 when this project was completed. We want to thank Sinai Robins
and Mich{\`e}le Vergne for numerous discussions of the topic. 
We acknowledge extremely helpful answers and comments on our questions on \url{mathoverflow.net}, 
in particular ones by David Eppstein, Dirk Lorenz, Igor Pak, David Speyer, and Gjergji Zaimi.
 Finally, the third author wants to
thank late Mikael Passare (who unfortunately left us so early) for  discussions
of the properties of Fantappi{\`e} transformation and for pointing out
reference~\cite{APS} in September 2011. 
\end{ack}

\section{Proving results on convex polytopes}\label{s2}
Following Brion-Lawrence-Khovanskii-Pukhlikov-Barvinok, see 
\cites{Bar2,BR,GLPR,La}, we
define for each vector $\z\in \bR^d$, the $j$-th \emph{axial} moment $\mu_j(\z)$ of a simple convex polytope $\cP$ with respect to $\z$ as
$$\mu_j(\z)=\int_{\cP}\lan \x,\z\ran ^jd\x.$$
We will use the following important statement, cf. e.g. \cite{BR}*{Theorem 10.5}.  
\begin{theorem}\label{th:BLKPB}
The moment $\mu_j(\z)$ satisfies
\begin{equation}\label{eq:BrBa}
\mu_j(\z)=\frac{(-1)^dj!}{(j+d)!}\sum_{\bv\in \Verti}\lan \bv,\z\ran ^{j+d}D_\bv(\z),
\end{equation}
where $D_\bv(\z):=\frac{|\det K_\bv|}{\prod_{j=1}^d\lan w_j(\bv),\z\ran }$, and $\z$ is an arbitrary vector for which the products $\prod_{j=1}^d\lan w_j(\bv),\z\ran $, $\bv\in\Verti$, do not vanish.   Moreover, the following  identities hold:
\begin{equation}\label{eq:Add}
\sum_{\bv\in \Verti}\lan \bv,\z\ran ^jD_\bv(\z)=0,
\quad\text{for $0\leq j\leq d-1$.}
\end{equation}
\end{theorem} 
\begin{proof}[Proof of Theorem~\ref{th:weight1}]
To  prove  \eqref{eq:Main}, consider the generating function  $$\Psi_\z(u)=\sum_{j=0}^\infty\frac{(j+d)!}{j!}\mu_j(\z)u^j,$$
where  $u\in \bR$. 
Formula \eqref{eq:BrBa} implies that $\Psi_\z(u)$ is rational. Indeed,
\begin{multline*}
	\Psi_\z(u)=\sum_{j=0}^\infty(-1)^d\sum_{\bv\in \Verti}\lan \bv,\z\ran ^{j+d}\frac{|\det K_{\bv}|u^j}{\prod_{k=1}^d\lan w_k(\bv),\z\ran }=\\
=(-1)^d\sum_{\bv\in \Verti}
\frac{\lan \bv,\z\ran ^d|\det K_{\bv}|}{\prod_{k=1}^d\lan w_k(\bv),\z\ran }
\sum_{j=0}^\infty\lan \bv,\z\ran ^ju^j=\sum_{\bv\in \Verti}
\frac{\lan \bv,\U\ran ^d|\det K_{\bv}|}{\prod_{k=1}^d\lan w_k(\bv),\U\ran }\cdot 
\frac{(-1)^d}{1-\lan \bv,\U\ran },
\end{multline*}
where $\U=u\z$. On the other hand, using  the multinomial coefficients  
$\binom{|J|}{J}=\frac{|J|!}{j_1 !\dots j_d!}$ of 
 multiindices  $J=(j_1,\dots,j_d)\vdash |J|$, one gets
\begin{multline*}
\int_\cP\lan \x,\z\ran ^jd\x=\int_{\cP}\left(\sum_{i=1}^dx_iz_i\right)^jd\x=\sum_{J\vdash j}\binom {j} {J}\z^J\int_{\cP}\bx^J d\x=\sum_{J\vdash j}\binom {j} {J}\z^J m_{_J}(\cP),
\end{multline*}
where $m_{_J}(\cP)=m_{_J}(\mu_\cP)$. 
Therefore,
\begin{multline*}
F_\cP(u\z):=F_\cP(uz_1,...,uz_d):=\sum_{j=0}^{\infty}\sum_{(j_1,...,j_d)\vdash j}\frac{(j+d)!}{j_1!\dots j_d!}m_{j_1,\dotsc,j_d}(uz_1)^{j_1}\cdots (uz_d)^{j_d}\\
%=\sum_{j=0}^{\infty}\frac{(j+d)!}{j!}\left(\sum_{J:=(j_1,...,j_d)\vdash j}\frac{j!}{j_1!\cdots j_d!}m_{j_1,...,j_d}u_1^{j_1}\cdots u_d^{j_d}\right)  =    \\ 
=\sum_{j=0}^{\infty}\frac{(j+d)!}{j!}\left[\sum_{J:=(j_1,...,j_d)\vdash j}\binom{j}{J}m_{_J}(\cP)\z^J\right]u^j = \sum_{j=0}^{\infty}\frac{(j+d)!}{j!}\mu_j(\z)u^j=\Psi_\z(u),%\\ 
%=\sum_{J:=(j_1,...,j_d)\vdash j}\frac{(j+d)!}{j!} \binom{j}{J}  m_{J}(u\z)^J.%=\sum_{J:=(j_1,...,j_d)\vdash j}\frac{(j+d)!}{j!} \frac{j!}   {j_1!\cdots j_d!}  m_{j_1,...,j_d}(uz_1)^{j_1}\cdots (uz_d)^{j_d}\sum_{j=0}^\infty \frac{(j+d)!}{j!}\mu_j(\z)u^j=\Psi_{\z}(u).
\end{multline*}
and \eqref{eq:Main} follows.  

In view of relations~(\ref{eq:Add}), the right-hand side of \eqref{eq:Main} can be rewritten as \eqref{eq:Mainsimp}.
%\begin{equation}\label{eq:Main2}
%F_{\cP}^1(u_1,...,u_d)=(-1)^d\sum_{v\in \Verti}\frac{|\det K_v |}{\prod_{j=1}^d\lan w_j(v),U\ran }\cdot \frac{1}{1-\lan v,U\ran }.
%\end{equation}
Indeed, writing $(1-\lan\bv,\bu\ran)^{-1}=\sum_{j=0}^\infty\lan\bv,\bu\ran^j$ and
expanding  \eqref{eq:Main} with respect  to $j$th powers of $\lan\bv,\bu\ran$, 
we see that \eqref{eq:Add} implies that for $j<d$ 
the sum of all terms  $\lan \bv,\U\ran ^j$ vanishes.
\end{proof}

\begin{proof}[Proof of Corollary~\ref{cor:simplex}]
Let $\Verti=(\bv_0,\bv_1,\ldots,\bv_d)$. Then
for each $j\neq i$, we have $w_j(\bv_i)=\bv_j-\bv_i$.  Hence $|\det K_{\bv_i}|$  does not depend upon $i$ 
and equals $d! \Vol(\Delta)$. The right-hand side of \eqref{eq:Main} becomes 
\begin{multline*}
(-1)^d d! \Vol(\Delta)\sum_{i=0}^d\frac{\lan \bv_i,\U\ran ^d}{\prod_{j=1}^d\lan w_j(\bv_i),\U\ran }\cdot \frac{1}{1-\lan \bv_i,\U\ran }=\\
=(-1)^d d! \Vol(\Delta)\sum_{i=0}^d\frac{\lan \bv_i,\U\ran ^d}{\prod_{j\ne i}\lan \bv_j-\bv_i,\U\ran }\cdot \frac{1}{1-\lan \bv_i,\U\ran }=\\
=(-1)^d d! \Vol(\Delta)\sum_{i=0}^d\frac{\zeta_i^d}{\prod_{j\ne i}(\zeta_j-\zeta_i)}\cdot \frac{1}{1-\zeta_i},
\end{multline*}
where $\zeta_i=\lan \bv_i,\U\ran $.
Computing the common denominator of the latter, we obtain
$$F_\Delta(\U)=\frac  {(-1)^d d! \Vol (\Delta)}   {\prod_{i=0}^{d}(1-\zeta_i)} 
\frac{\sum\limits_{i=0}^d\left[\prod\limits_{k>l, k\ne i\ne l}(\zeta_k-\zeta_l)\prod\limits_{j\ne i}(1-\zeta_j)\right](-1)^i \zeta_i^d}{\prod_{s>t}(\zeta_s-\zeta_t)}.$$
It is convenient to introduce one more linear form $\zeta_{d+1}:=1$, so that the last expression reads as
\begin{equation}\label{eq:fdelta}
F_\Delta(\U)=(-1)^d d! \Vol (\Delta) 
\frac{\sum\limits_{i=0}^d\left[\prod\limits_{\substack{d+1\ge k>l\ge 0,\\ k\ne i\ne l}}(\zeta_k-\zeta_l)\right](-1)^i \zeta_i^d}{\prod\limits_{d+1\ge s>t\ge 0}(\zeta_s-\zeta_t)}.
\end{equation}
To complete the proof, we notice that 
\begin{multline*}
0 = \det \begin{pmatrix}
    1 & 1 &\ldots & 1 & 1 \\
    \zeta_0 & \zeta_1 & \ldots & \zeta_d & 1 \\
    \zeta_0^2 & \zeta_1^2 & \ldots & \zeta_d^2 & 1 \\
    \ldots \\
    \zeta_0^d & \zeta_1^d & \ldots & \zeta_d^d & 1 \\
    \zeta_0^d & \zeta_1^d & \ldots & \zeta_d^d & 1 \\
\end{pmatrix}=1\cdot \det \begin{pmatrix}
    1 & 1 &\ldots & 1  \\
    \zeta_0 & \zeta_1 & \ldots & \zeta_d \\
    \zeta_0^2 & \zeta_1^2 & \ldots & \zeta_d^2 \\
    \ldots \\
    \zeta_0^d & \zeta_1^d & \ldots & \zeta_d^d \\
\end{pmatrix}+\\+ (-1)^{d+1}\sum_{i=0}^{d}\zeta_i^d(-1)^i\cdot\det \begin{pmatrix}
    1 & \ldots & 1 & 1 & \ldots & 1 \\
    \zeta_0 & \ldots & \zeta_{i-1} & \zeta_{i+1} & \ldots & \zeta_{d+1} \\
    \zeta_0^2 & \ldots & \zeta_{i-1}^2 & \zeta_{i+1}^2 & \ldots & \zeta_{d+1}^2 \\
    \ldots \\
    \zeta_0^d & \ldots & \zeta_{i-1}^d & \zeta_{i+1}^d & \ldots & \zeta_{d+1}^d \\
\end{pmatrix}=\\
=\prod_{d\ge k>l\ge 0}(\zeta_k-\zeta_l)+(-1)^{d+1}\sum_{i=0}^{d}\zeta_i^d(-1)^i\cdot\left[\prod\limits_{\substack{d+1\ge k>l\ge 0,\\ k\ne i\ne l}}(\zeta_k-\zeta_l)\right].
\end{multline*}
Indeed, the first matrix has two identical rows and thus vanishing determinant, 
which we expand with respect to the last row.
The last equality is the standard formula for the Vandermonde determinant. Thus we have 
$$
\sum_{i=0}^{d}\zeta_i^d(-1)^i\cdot\left[\prod\limits_{\substack{d+1\ge k>l\ge 0,\\ k\ne i\ne l}}(\zeta_k-\zeta_l)\right]
= 
(-1)^d\prod_{d\ge k>l\ge 0}(\zeta_k-\zeta_l).
$$
Now we plug this formula into \eqref{eq:fdelta} and get 
\[
F_\Delta(\U)= d!\Vol(\Delta) \frac{1}{\prod\limits_{t=0}^{d}(1-\zeta_t)}= \frac{d!\Vol(\Delta)}{\prod\limits_{\bv\in\Verti}(1-\lan \bv,\U\ran)}. \qedhere
\]
\end{proof}

\begin{lemma}\label{lem:diff_h}
Let $\bu=(u_1,\dots,u_d)$ and $\x=(x_1,\dots,x_d)$ be formal variables, and $\ell\in\bR$ .
Then
\begin{equation}\label{eq:diff_h}
\left(\sum_k u_k\frac{\partial}{\partial u_k}+\ell\right)\circ (1-\lan \x,\U\ran)^{-\ell}=
\ell (1-\lan \x,\U\ran)^{-\ell-1}. 
\end{equation}
\end{lemma}
\begin{proof}
Note that $u_k\frac{\partial}{\partial u_k}\circ (1-\lan \x,\U\ran)^{-\ell}=
x_k u_k \ell  (1-\lan \x,\U\ran)^{-\ell-1}$. Thus 
\begin{multline*}
\left(\sum_k u_k\frac{\partial}{\partial u_k}+\ell\right)\circ (1-\lan \x,\U\ran)^{-\ell}=
\ell  (1-\lan \x,\U\ran)^{-\ell-1} \lan \x,\U\ran + \ell  (1-\lan \x,\U\ran)^{-\ell}\\
=\ell (1-\lan \x,\U\ran)^{-\ell-1}.\qedhere
\end{multline*}
\end{proof} 

\begin{proof}[Proof of \eqref{eq:Mainint}]
For a $d$-variate polynomial $g(\z)$, we denote by $g\left(\bu\frac{\partial}{\partial \bu}\right)$ the differential operator 
$g\left(u_1\frac{\partial}{\partial u_1},\dots,u_d\frac{\partial}{\partial u_d}\right)$. We use the  identity
\begin{equation}\label{eq:iddif}
g\left(\bu\frac{\partial}{\partial \bu}\right)
\circ \sum_{I} a_{_I}\x^I\U^I= \sum_{I\geq 0} g(I) a_{_I} \x^I\U^I,
\end{equation}
which holds for any formal $d$-variate power series $\sum_I a_I \x^I\U^I$ and any $d$-variate polynomial $g(\z)$. (It can be easily verified for monomial $g(\z)$ and then extended by linearity.)  
Setting 
$h(\z):=\prod_{\ell=1}^d\left(\sum_{k=1}^d z_k +\ell\right),$  notice that $h(I)=(|I|+1)(|I|+2)\cdots (|I|+d)$. 
Now using \eqref{eq:iddif} together with  the obvious identity:   
$$(1-\lan\x,\U\ran)^{-1}=\sum_{I\ge 0}\binom{|I|}{I} \x^I\U^I,$$  
 one obtains
\begin{multline*}
F_{\mu}(\U):=\sum_{I\ge 0}\binom{|I|+d}{I} m_{I}(\mu) \U^I =\sum_{I\ge 0}h(I)\binom{|I|}{I} m_{I}(\mu) \U^I \\ 
=\int_{\bR^d} \sum_{I\ge 0}h(I)\binom{|I|}{I} \x^I\U^I  d\mu(\bx)=
\int\limits_{\bR^d} h\left(\U\frac{\partial}{\partial\U}\right)\circ \sum_{I\ge 0}\binom{|I|}{I} \x^I\U^I  d\mu(\x)\\
=\int\limits_{\bR^d} h\left(\U\frac{\partial}{\partial\U}\right)\circ 
\frac{ d\mu(\x)}{1-\lan\x,\U\ran}
=\int\limits_{\bR^d} \frac{d!\  d\mu(\x)}{(1-\lan\x,\U\ran)^{d+1}},
\end{multline*}
where in the final derivation we repeatedly made use of \eqref{eq:diff_h}, for $1\leq \ell\leq d.$ 
\end{proof}

\begin{rema+}\label{rem:Fanta}
Another point of view on \eqref{eq:Mainint} is that it is the result of the
application of the differential operator 
$g\left(\U\frac{\partial}{\partial\U}\right)$ to the integral 
transformation 
$\int_{\bR^d}\frac{d\mu(\x)}{1-\lan \x,\U\ran}$ of the measure $\mu$ (also known as the Fantappi\`{e}  transform of $\mu$); see e.g. \cite{MR2274973}.

In \cite{PS12} a similar idea was applied to the harmonic polygonal measures in the plane.
\end{rema+}

\begin{proof}[Proof of Theorem~\ref{th:weight}]
Assume first that $\rho(x_1,...,x_d)={\bf x}^{K}=x_1^{k_1}\cdots x_d^{k_d}$ is a monomial and consider
$$\rho\left (\frac{\partial}{\partial {\bu}} \right)\circ F_{\mu}({\bu})=\frac{\partial^{|{K}|}}{\partial u_1^{k_1}\cdots \partial u_d^{k_d}}\circ \sum_{I=(i_1,\dots,i_d)\geq 0}
\frac{(|I|+d)!}{i_1!\cdots i_d!} m_I(\mu) \bu^I.$$
One gets
\begin{multline*}
\frac{\partial^{|K|}}{\partial {\bu}^{K}}\circ F_{\mu}({\bu})=\sum_{I=(i_1,...,i_d)\geq 0} 
\left(\frac{(|I|+|{ K}|+d)!}{\prod\limits_{j=1}^d (i_j+k_j)!}\prod_{j=1}^d\frac{(i_j+k_j)!}{i_j!}\right)m_{I+K}(\mu)\bu^I\\
=\sum_{I=(i_1,\dots,i_d)\ge 0}\frac{(|I|+d+|{K}|)!}{\prod_{j=1}^d i_j!}m_{I}(\bx^K\mu)\bu^I.
\end{multline*}
Observe that the normalizing coefficients of $m_I(\bx^K\mu)$ in the latter 
expression depend only on $I$ and $|K|$ but not on particular entries of  $K$. 
Therefore for an arbitrary homogeneous $\rho$ of degree $\delta$,   one gets  by additivity
\[ \rho\left (\frac{\partial}{\partial {\bu}} \right)\circ F_{\mu}({\bu})=
\sum_{I=(i_1,\dots,i_d)\ge 0}\frac{(|I|+d+\delta)!}{\prod\limits_{j=1}^d i_j!}m_{I}(\rho\mu). \]
This shows \eqref{eq:gen_ii}.
Repeated application of \eqref{eq:diff_h}, for $d+1\leq \ell\leq d+\delta$, 
 to the integral representation \eqref{eq:Mainint}, respectively, to the representation 
\eqref{eq:nmomgf}, of
$F_{\rho\mu}(\bu)$ implies \eqref{eq:gen_iii}, respectively, \eqref{eq:gen_i}.
\end{proof}

\section{Inverse moment problem for strongly non-degenerate $S$}\label{s3}

\begin{proof} [Proof of Lemma~\ref{lm:inv}]

We prove first that the tangent cone at any non-vertex allows a decomposition into  convex polytopal cones each having a  
translation-invariant direction.

Let $\bv$ be a point in $\cP$ which is  not a vertex.
Then there is a dissection $\cT$ of $\cP$ such that $\bv$ is not a vertex of any simplex of $\cT$. 
Let $U$ be the set of simplices $S_u$ of $\cT$ with closures containing $\bv$.
Take the dissection of the tangent cone $T_v(\cP)$ into the tangent cones
to simplices from $U$,  $T_\bv(\cP)=\cup_{u\in U} T_\bv(S_u)$. 
Clearly, every subcone $T_v(S_u)$ contains a translation-invariant direction 
(any direction parallel to the minimal face containing $\bv$). %Claim is proved.

Vice versa, to prove the converse implication,  let us take a dissection of the tangent cone 
$T_\bv(\cP)$ into a disjoint union of  convex polytopal cones $Q_1,\dots,Q_k$. By definition of the tangent cone and 
since $\cP$ can be represented as a finite union of simplices, we obtain that any sufficiently small 
neighborhood of $\bv$ in the tangent cone $T_\bv(\cP)$ is a neighborhood  of $\bv$ in the entire $\cP$.
Consider the parallelepiped $\Boxe$ centered at $\bv$ that is the $\veps$-ball centered at $\bv$, in the $L_1$-norm. 
Note that each convex polytopal set $Q_i\cap \Boxe$ can be decomposed into a union of simplices that do not contain $\bv$ as a vertex.

Further notice that the set $\cP\setminus\Boxe$ can be represented as a finite disjoint union of simplices, 
since $\Boxe$ is the intersection of a finite number of half-spaces and $\cP$ is a disjoint union of simplices.
Clearly, every simplex in this union should not have $\bv$ as a vertex. 
Now combining the dissections of each $Q_i\cap \Boxe$ and $\cP\setminus\Boxe$ we obtain the required dissection of $\cP$.
\end{proof}

%If $\bv$ is not a vertex of a generalized polytope $\cP$ then the tangent cone $T_p(\cP)$ to $\cP$ at $\bv$ allows such a decomposition. Indeed, let $\cT$ be a dissection of $\cP$ such that $\bv$ is not a vertex of $\cT$. The closure of the union of the tangent cones  to all simplices of $\cT$ whose closure contains $\bv$ coincides with $T_p(\cP)$. But the tangent cone to every simplex whose closure contains $\bv$ clearly  has a translation-invariant direction (parallel to minimal face of the closure containing $\bv$).

%Vice versa, let we have such decomposition. According to the definition of tangent cone and since $\cP$ is a finite union of simplices a small neighborhood of tangent cone at $\bv$ lies entirely in $\cP$.

%If $Q$ is a subcone of this decomposition and $s$ is the translation-invariant direction of $\dim k<d$ ($s$ necessarily contains $\bv$) take a small simplex $\delta$ in $s$. Take a polyhedral link $L$ of $s$ in $Q$ of $\dim d-k-1$. 
%Each simplex in $L$ and $\delta$ generates a $d$-dimensional simplex in $Q$. The union of all these simplices form a triangulated full-dimensional neighborhood $N_Q $ of $\bv$ in $Q$ where $\bv$ is not a vertex. The union of all $N_Q$ makes a neighborhood $N$ of $\bv$ in $P$. To finish a proof it is enough to extend this dissection of $N$ to the dissection of the whole $\cP$ in arbitrary way.

\begin{proof}[Proof of  Proposition~\ref{prop:gener}] 
We begin by considering the case $\rho\equiv 1$.
Let $\cT$ be a dissection of $\cP$ with vertices $\Verti(\cT)$. 
Corollary~\ref{cor:simplex} implies
that $F_\cP(\bu)$ has a  denominator
dividing $g_{\cT}(\bu)=\prod_{\bv\in \Verti(\cT)}(1-\lan \bv,\bu\ran)$.
Take $\bv_1\in \Verti(\cT)\setminus\Verti(\cP)$. Then there exists another
dissection $\cT'$ such that
$\bv_1\not\in\Verti(\cT')$.
Expressing  $F_\cP(\bu)$  as ratios of polynomials, we have
$$F_\cP(\bu)=\frac{f_{\cT}(\bu)}{h_{\cT}(\bu)(1-\lan \bv_1,\bu\ran)}=\frac{f_{\cT'}(\bu)}{g_{\cT'}(\bu)},\quad\text{where }
g_{\cT}(\bu)={h_{\cT}(\bu)(1-\lan \bv_1,\bu\ran)}.$$
Here $g_{\cT'}$ is not divisible by $1-\lan \bv_1,\bu\ran$, by the choice
of $\cT'$. Thus
$f_{\cT}$ is divisible by $1-\lan \bv_1,\bu\ran$, and can be canceled out
in the expression for
$F_\cP(\bu)$.

The case of arbitrary homogeneous $\rho$ follows immediately by applying Theorem~\ref{th:weight} to the 
already covered case $\rho\equiv 1$. 
\end{proof}

\begin{proof}[Proof of Proposition~\ref{pr:generic}] First we show that for an arbitrary  finite spanning set $S\subset \bR^d$, the space $\M^\De(S)$ has a basis 
of $d$-dimensional simplices containing a fixed vertex $\bv\in S$.  In particular, the set of all $d$-dimensional simplices containing  $\bv$ spans $\M^\Delta(S)$ but is not necessarily a basis.  Consequently,  their Fantappi\`{e} transformations spans $\F^\Delta(S)$. %Indeed, fixing any vertex $\bv$ take any simplex $\Delta=span(\bv_{i_1},...,\bv_{i_{d+1}})$ with vertices in $S$. Without loss of generality we may assume  that $\bv=\bv_N$.  Assume that $\Delta$ does not contain $\bv$. 
 The following result is formulated as Theorem 4.2 of \cite{Al1} and in a different form  in \cite{AlGelZel}. (We omit the proof of this statement here.)
 
Given two points $p$ and $q$ and a  set $M$ in $\bR^d$, we say that $q$ is {\it visible} from $p$ with respect to $M$ if the line segment $pq$ is disjoint from $M$. 

\begin{lemma}\label{lm:simplex}
Given a $d$-dimensional simplex $\sigma\subset \bR^d$, denote by $\Verti(\sigma)$ the set of vertices of $\sigma$. 
Let $\sigma^0$ be the interior of $\sigma$. Let $p$ be any point in $\bR^d$ and let $Q^+$ (resp. $Q^-$) be the set  
of all $(d-1)$-dimensional faces of $\sigma$ which are  visible (resp. not visible) from $p$ with respect to $\sigma^0$.  
Then the standard measures of all $d$-dimensional simplices with vertices in 
$\Verti(\sigma)\cup\{ p\}$ satisfy
$$\mu_{\sigma}=\sum_{\sigma_i\in Q^+}\mu_{\sigma_i,p}-\sum_{\sigma_i\in Q^-}\mu_{\sigma_i,p},$$
where $\mu_{\sigma_i,p}$  is the standard measure of the $d$-dimensional simplex spanned by the vertices of $\sigma_i$ and the point $p$. 
\end{lemma}

\begin{rema+}
If $\sigma_{i,p}$ is a degenerate simplex, i.e., $p$ lies in the hyperplane spanned by  $\sigma_i$, we simply exclude the corresponding 
term $\mu_{\sigma_{i,p}}$ from the above formula.
\end{rema+}

To prove Proposition~\ref{pr:generic}, we need to show that $\F^\De(S)$ coincides with $\mathfrak{Rat}(S)$ if and only if $S$ is strongly non-degenerate. Indeed,  $\F^\De(S)\subseteq \mathfrak{Rat}(S)$ for an arbitrary spanning $S$, by Proposition~\ref{prop:gener}. The Fantappi\`{e}  transform $F_\mu: \M^\De(S)\to \F^\De(S)$ is a linear isomorphism which implies that $\dim \M^\De(S)=\dim \F^\De(S)$.  By Lemma~\ref{lm:simplex}, the space $\M^\De(S)$ is spanned by the standard measures $\mu_\De$ of the set $\B_{i}$ of all $d$-dimensional simplices containing the fixed vertex $\bv_{i}$. Let us fix the vertex $\bv_{_N}$ and consider the set $\B_{_N}$. For $S$ strongly non-degenerate, the cardinality of $\B_{_N}$ equals $\binom {N-1}{d}$. 

Now we show that $\F^\De(S)=\mathfrak{Rat}(S)$, where $\mathfrak{Rat}(S)$ has the dimension 
$\binom {N-1}{d}$, as it is isomorphic to the space $Pol(N-d-1, d)$ of all  $d$-variate polynomials of degree at most $N-d-1$.
% or equivalently to the space of homogeneous polynomials of degree $N-d-1$ in $d+1$ variables. 
This would immediately imply that the standard measures of simplices in $\B_{_N}$ are linearly independent. 

\begin{lemma}
\label{lm:product}
If $S$ is strongly non-degenerate, then $\F^\De(S)=\mathfrak{Rat}(S)$.
\end{lemma}
\begin{proof}
We recall that $\F^\De(S)$ comprises all linear combinations of the 
rational functions $$\frac{d!\Vol(\bv_{i_1},\dots,\bv_{i_d},\bv_{_N})}{(1-\lan\bv_{i_1},\bu\ran)\cdot\dots\cdot(1-\lan\bv_{i_{d}},\bu\ran)\cdot(1-\lan\bv_{_N},\bu\ran)}.$$ 
For each term $1-\lan\bv_{i},\bu\ran$, we consider a (homogeneous) linear form $l_i(u_0,\bu)=u_0-\lan\bv_{i},\bu\ran$ in $d+1$ variables $u_0,\dots,u_d$.
Set $n=N-1$ where $n\ge d+1$. For the $n$-tuple $\LL=\{l_1,l_2,....,l_{n}\}$ of linear $(d+1)$-variate forms, 
let $V_\LL$ be the linear span of all possible products of the form $l_{i_1} l_{i_2}  \dots  l_{i_{n-d}},\; 1\le i_1<i_2<...<i_{n-d}\le n$.
Observe that $V_\LL$ is the space of all numerators that one can obtain in $\F^\De(S)$. We need to show that $V_\LL$ contains $HPoly(n-d,d+1)$, 
the space of all  $(d+1)$-variate homogeneous polynomials of degree $n-d$.
Recall that  any $d+1$-tuple of linear forms $l_{i_1},\dots,l_{i_{d+1}}$ is linearly independent  due to the strong degeneracy assumption. Thus 
we can express each single variable $u_0,\dots, u_d$ as a linear combination of these forms. Since $V_\LL$
contains all  products $l_{i_1} l_{i_2}  \dots  l_{i_{n-d-1}} l_j$, where $j\in \{1,\dots,n\}\setminus\{i_1,\dots,i_{n-d-1}\}$,
we conclude that $V_\LL$ contains all homogeneous polynomials of the form 
\[
l_{i_1} l_{i_2}  \dots  l_{i_{n-d-1}} u_k,
\quad\text{ for  $0\le k\le d$.}
\]
From that we deduce that $V_\LL$ contains all 
homogeneous polynomials of the form $l_{i_1} l_{i_2}  \dots  l_{i_{n-d-2}} u_k u_j$, 
where $j,k \in [n]\setminus\{i_1,\dots,i_{n-d-1}\}$. Continuing along the same lines, we derive by 
induction that $V_\LL$  contains $HPoly(n-d,d+1)$.
\end{proof}

For an arbitrary spanning $S$, the  cardinality of $\B_{_N}$ is at most $\binom {N-1}{d}=\dim \mathfrak{Rat}(S)$.  Furthermore,  if $S$ is not strongly 
non-degenerate the cardinality of $\B_{_N}$ is strictly smaller than $\binom {N-1}{d}$, as there will be linear dependencies
among the standard measures on the simplices in $B_{_N}$. Therefore, $\dim \F^\De(S)<\dim \mathfrak{Rat}(S)$. 
\end{proof}

We define the square matrix $Mat_{\LL}$ of size $\binom{n}{d}$ with entries being coefficients of the above products of linear forms w.r.t. the standard monomial basis in $HPol(n-d,d+1)$.  %We need the next intriguing  statement.

\begin{lemma}\label{rm:product}
 The determinant of $Mat_{\LL}$ is proportional to the product of the determinants of all $(d+1)$-tuples $(l_{i_1},l_{i_2},\ldots, l_{i_{d+1}}),\; i_1<i_2<\ldots <i_{d+1}$. (By the determinant of a $(d+1)$-tuple of vectors in $\bR^{d+1}$ with a fixed basis we mean the  determinant of the matrix formed by the coordinates of these vectors in a chosen basis.) 
\end{lemma} 

\begin{proof}  Indeed, $\det(Mat_{\LL})$ is a form of degree $(d+1)\binom{n}{d}$ in the coefficients of the linear forms $l_1,\ldots, l_n$. Thus the product  $\prod_{i_1,\ldots, {i_{d+1}}}\det(l_{i_1},l_{i_2},\ldots, l_{i_{d+1}})$ has the same degree as  $\det(Mat_{\LL})$. Therefore it suffices to show that $\det(Mat_{\LL})$ vanishes as soon as some of $\det(l_{i_1},l_{i_2},\ldots, l_{i_{d+1}})$ vanishes. (Observe that all polynomials $\det(Mat_{\LL})$ are coprime.) 
Without loss of generality,  assume that $l_1$ is a linear combination of $l_2,\ldots,l_{d+1}$. But then the column of $Mat_{\LL}$ corresponding to the $(n-d)$-tuple $(1,d+2,d+3,\ldots, n)$ will be a linear combination of those corresponding to $(2,d+2,d+3,\ldots, n),$ \ldots $(d+1,d+2,d+3,\ldots, n)$. 
\end{proof}

\begin{proof}[Proof of Corollary~\ref{cor:generic}] 
As we mentioned above, $\M(S)$ is isomorphic to $\F(S)$ and, analogously, 
$\M^\Delta(S)$ is isomorphic to $\F^\Delta(S)$. Thus, if we prove the equality $ \F(S)=\F^\Delta(S)$, then we get $\M(S)=\M^\Delta(S)$. 
By Lemma~\ref{lm:product}, the space $\F^\Delta(S)$ coincides with the linear space of all rational functions with
the numerator an arbitrary polynomial of degree at most $N-d-1$ and the denominator $\Phi_S(\bu)$ equal to the product of all linear forms dual to all vertices in $S$.  
By Proposition~\ref{prop:gener} an arbitrary function in  $\F(S)$ is a rational function with denominator of desired form  and numerator of degree at most $N-d-1$, for obvious reasons---take an arbitrary dissection and sum over its simplices. Since all such functions are already in $\F^\Delta(S)$ we are done.  \end{proof} 
 
\begin{proof}[Proof of Theorem~\ref{th:inverse}]  Given a strongly non-degenerate set $S=\{\bv_1,\ldots , \bv_{N-1}, \bv_{N}\}$ and the Fantappi\`{e} transform $R(\bu)=P(\bu)/\Phi_S(\bu)$, 
where $\Phi_S(\bu)=\prod_{j=1}^Nl_{j}(\bu)$,  we want to solve the inverse moment problem. (It is easy to obtain $P(\bu)$ from information on the moments of order at most $N-d-1$ using Lemma~\ref{lm:McL}.) %Recall that  $N=n+1$ in notation of Introduction.)  

To solve the latter inverse problem  using Corollary~\ref{cor:simplex}, we need to find an appropriate set of weights $\bw=\{w_{i_1,\ldots, i_d}\},\; i_1<i_1<\ldots <i_d$, where $w_{i_1,\ldots, i_d}$ is the weight (recall
Definition~\ref{eq:weight}) of the $d$-dimensional simplex $\conv(\bv_{i_1},\bv_{i_2},\ldots , \bv_{i_d}, \bv_{N})$ so that 
$$\sum_{i_1<i_2<\ldots <i_d}\frac{w_{i_1,\ldots, i_d}}{l_{i_1}(\bu) l_{i_2}(\bu) \dots  l_{i_d}(\bu) l_{_N}(\bu)}=\frac{P(\bu)}{l_{1}(\bu) l_{2}(\bu) \dots  l_{_N}(\bu)}.$$
%(Recall that $w_\De=d!\int_{\De}d\mu$.) 
Clearing the denominators, we get the equation 
$$\sum_{i_1<i_2<\ldots <i_d}{w_{i_1,\ldots, i_d}}{l_{j_1}(\bu) l_{j_2}(\bu) \dots  l_{j_{N-d-1}}(\bu) l_{N}(\bu)}={P(\bu)},$$
where  $\{j_1,\ldots j_{N-d-1}\}=\{1,2,\ldots, N-1\}\setminus \{i_1,\ldots , i_d\}$. The latter equation is obviously equivalent to the system of linear equations
$$Mat_{S}\cdot \bw=(p_{I_1},\dots,p_{I_t}), \quad\text{where\ } P(\bu)=\sum_I p_I \bu^I,$$
and $\bw=\{w_{i_1,\ldots, i_d}\}$ is the vector  consisting of the weights of all simplices containing $\bv_{N}$. 
\end{proof}

Theorem~\ref{th:inverse} solves the inverse moment problem for strongly non-degenerate spanning set $S$. We can make this solution more explicit   by giving a closed formula for  the inverse matrix $Mat_{S}^{-1}$.  To do this, we 
introduce an extra variable $u_0\in\bR$ and identify the space $Pol(N-d-1,d)$ with 
the space $HPol(N-d-1, d+1)$ of homogeneous forms of degree  $N-d-1$ in $d+1$ variables $(u_0,u_1,\dots, u_d)$. 
We homogenize each linear form $l_i(\bu)$ in $\LL$ as $l_i(\bu,u_0)=u_0-\lan\bv_i,\bu\ran$. (The matrix $Mat_{S}$ remains unchanged.) 

We also need the following $(d+1)\times (N-1)$ matrix $\bL$ 
\begin{equation}
\label{eq:Lmatrix}
\bL = \bordermatrix{~ & l_1 & l_2 & \dots & l_{N-1} \cr
                  u_0 & 1 & 1 & \dots & 1 \cr
                  u_1 & -\lan\bv_1, e_1 \ran & -\lan\bv_2, e_1 \ran & \dots & -\lan\bv_{N-1}, e_1 \ran \cr
                  ~\vdots & \vdots & \vdots & & \vdots \cr
                  u_d & -\lan\bv_1, e_d \ran & -\lan\bv_2, e_d \ran & \dots & -\lan\bv_{N-1}, e_d \ran \cr
                  }
\end{equation}
associated with $Mat_{S}$. 
For all possible subsets  $\bi[d]=\{i_1,\dots, i_d \},$  $i_1<i_2<\dots<i_d$ of $d$ distinct columns of $\bL$,  consider 
the linear in $\bu$ function $\bL_{\bi[d]}(\bu)$  given by:
\begin{equation}
\label{eq:minor}
\bL_{\bi[d]}(\bu)=\text{det}
      \begin{bmatrix}
      u_0    & 1                       & \dots  &  1                      \\
      u_1    & -\lan\bv_{i_1}, e_1\ran & \dots  & -\lan\bv_{i_d}, e_1\ran \\
      \vdots &                         & \dots  &                         \\
      u_d    & -\lan\bv_{i_1}, e_d\ran & \dots  & -\lan\bv_{i_d}, e_d\ran
      \end{bmatrix}.
\end{equation}
Denote by  $\bL_{\bi[d],j}$ the coefficient of $u_j$ in the linear form 
$\bL_{\bi[d]}(\bu)$. For each $1\le j\le N-1$,  define $\bL(j,\bi[d])$ as
\begin{equation}
\label{eq:fullminor}
\bL(j,\bi[d])=\text{det}
      \begin{bmatrix}
      1                      & 1                       & \dots  &  1                      \\
      -\lan\bv_{j}, e_1\ran  & -\lan\bv_{i_1}, e_1\ran & \dots  & -\lan\bv_{i_d}, e_1\ran \\
      \vdots                 &                         & \dots  &                         \\
      -\lan\bv_{j}, e_d\ran  & -\lan\bv_{i_1}, e_d\ran & \dots  & -\lan\bv_{i_d}, e_d\ran
      \end{bmatrix}.
\end{equation}
Note that if $j\notin \bi[d]$ then $\bL(j,\bi[d])\neq 0$, as by the assumption  of  strong non-degeneracy of $S$  the corresponding $d+1$ linear forms are linearly independent. On the other hand, if $j\in\bi[d]$, then we have $\bL(j,\bi[d])= 0.$

The matrix $Mat_{S}^{-1}$ has the following explicit description. 

\begin{lemma}\label{lm:mat} For each $(N-d-1)$-tuple of  forms $\{l_{j_1},l_{j_2},\dots,l_{j_{N-d-1}}\}$, set 
$\bi[d]=\{i_1,\dots,i_d\}=[N-1]\setminus\{j_1,\dots,j_{N-d-1}\},$ where $i_1<i_2\dots<i_d.$ Then, 
\begin{equation}
\label{eq:inverse}
Mat_{S}^{-1}= \bordermatrix{
                  ~                        &    \dots &  u_0^{n_0}u_1^{n_1}\dots u_d^{n_d}                 & \dots        \cr
                  \quad\vdots              &          &  \quad\vdots                                       &  \cr
                  l_{j_1}\dots l_{j_{N-d-1}} &    \dots &  \frac{\prod\limits_{j=0}^{d}\bL_{\bi[d],j}^{n_j}}
                                                         {\prod\limits_{k=1}^{N-d-1} \bL(j_k,\bi[d])}       & \dots \cr
                  \quad\vdots              &          &  \quad\vdots                                       &  \cr
                }.
\end{equation}
\end{lemma}

\begin{proof}[Proof of Lemma~\ref{lm:mat}]

In order to show that $Mat_{S}^{-1}$ defined by \eqref{eq:inverse} is indeed the inverse of $Mat_{S}$ we  need to verify 
that $Mat_{S}^{-1}\cdot Mat_{S}$ is the identity  operator on $V_\LL$. 

Let $\bfe'$ be the standard basis vector of $V_\LL$ 
corresponding to the product of linear forms $l_{j'_1}\dots l_{j'_{N-d-1}}$. 
Then $Mat_{S}\cdot\bfe'$ is the vector consisting  of the monomial coefficients of the homogeneous form 
$l_{j'_1}\dots l_{j'_{N-d-1}}$ in the variables $u_0,\dots,u_d$.
Let $\bfe^{T}$ be the row vector of $Mat_{S}^{-1}$ corresponding to the product 
$l_{j_1}\dots l_{j_{N-d-1}}$. We note that in $\bfe^{T}\cdot \left(Mat_{S}\cdot\bfe'\right)$ one can
factor out the common denominator $\prod\limits_{k=1}^{N-d-1} \bL(j_k,\bi[d])$ 
of all fractions in $\bfe^T$; the remaining factor  is of the form 
\begin{equation}
\label{eq:Scalar_Product}
\sum\limits_{\substack{I=(n_0,\dots,n_d) \\ |I|= N-d-1}} Mat_{S}[\bu^{I},\bfe']\prod\limits_{j=0}^{d}\bL_{\bi[d],j}^{n_j}
=l_{j'_1}(\bL_{\bi[d]}(\bu))\dots l_{j'_{N-d-1}}(\bL_{\bi[d]}(\bu)).
\end{equation}
Note that $l_{j'_k}(\bL_{\bi[d]}(\bu))=\bL(j'_k,\bi[d])$ for $1\leq k\leq N-d-1$, 
i.e., this holds for all terms in the product on the right-hand side                                      
of \eqref{eq:Scalar_Product}.  
Hence, if $\bfe\neq\bfe'$ then 
$\bfe^{T}\cdot \left(Mat_{S}\cdot\bfe'\right)=0$, as among $j'_1,\dots ,j'_{N-d-1}$ one can find $j'\in\bi[d]$ 
with $\bL(j',\bi[d])=0$. On the other hand, if $\bfe$ and $\bfe'$ coincide, then the right-hand side of
\eqref{eq:Scalar_Product} is
equal to $\prod\limits_{k=1}^{N-d-1} \bL(j_k,\bi[d])$. Dividing by the common denominator of the fractions in 
$\bfe$, we obtain $\bfe^{T}\cdot \left(Mat_{S}\cdot\bfe\right)=1.$
\end{proof}

\section{Inverse moment problem for  weakly non-degenerate $S$}\label{s4}
Given an arbitrary spanning set $S=\{\bv_1,\bv_2,\ldots, \bv_N\}$, consider the linear space $\Theta(S)\subseteq Pol(N-d-1,d)$  spanned by all products $l_{j_1}l_{j_2}\ldots l_{j_{N-d-1}},\; j_1<j_2<\ldots <j_{N-d-1}$. %where $Pol(N-d-1,d)$ is the space of all non-homogeneous polynomials of degree at most $N-d-1$.  
The next statement explains why we can extend our solution of the inverse moment problem from the case of strongly  non-degenerate $S$ to the case of weakly non-degenerate $S$. 
\begin{lemma} 
\label{lem:generic} $\Theta(S)=Pol(N-d-1, d)$ if and only if 
$S$ is  weakly non-degenerate, i.e.,  each $(d+2)$-tuple of  points of $S$ is spanning. 
\end{lemma}
\begin{proof}
%%%%%%%%%%%%%%%%%%%%%%%%%%%%%%%%%%%%%%%%%%%%%%%%%%%%%%%%%%%%%%%%%%%%%%%%%%%%%%%%%%%%%%%%%%%%%%%%
We have $N$ (non-homogeneous) linear forms $l_1\dots,l_N$ in  variables $\bu=(u_1,\dots, u_d)$ and the  linear 
space $V_\LL$ spanned by  all possible products of $(N-d-1)$-tuples of  distinct forms. 
We need to investigate whether  $V_\LL$ coincides with $Pol(N-d-1, d)$. 
Homogenizing, we consider the same question for the linear  homogeneous forms  and the homogeneous 
polynomials of degree  $N-d-1$ in  variables $(u_0,u_1,\dots, u_d)$. 

First  assume that there are $d+2$ linear forms $l_1,\dots,l_{d+2}$ which are not spanning. 
Then one can find a non-zero vector $\z_0\in\RR^{d+1}$, such that $l_1(z)=\dots=l_{d+2}(z)=0$. Note that
each product of $N-d-1$ different forms chosen from  $l_1,\dots, l_N$ contains at least one form among $\{l_1,\dots,l_{d+2}\}$.
Therefore any linear combination of products of $N-d-1$ forms vanishes at $\z_0$. Thus $V_\LL$ cannot coincide with  
$HPol(N-d-1,d+1)$.

Conversely, assume that every $(d+2)$-tuple of  distinct forms among $l_1,\dots, l_N$ is spanning.
First, we notice that $HPol(N-d-1,d+1)$   can be spanned 
by the all possible products of $N-d-1$ linear forms (not necessarily pairwise distinct). Indeed, since 
first $d+2$ forms span the dual space of  $\RR^{d+1}$, we can express each variable $x_i$ as a linear combination 
of these forms. Therefore every monomial of degree $N-d-1$ can be expressed as a linear combination
of products of $N-d-1$ forms. 

Now we show that each product of $N-d-1$, not necessarily distinct, forms can be expressed as 
a linear combination of the products of distinct ones. Assume the contrary and consider 
monomials $l_1^{i_1}\dots l_N^{i_N}$ of degree $N-d-1$ which cannot be expressed as a linear combination
of products with all distinct forms. Among those monomials we take a monomial $\bm= l_1^{k_1}\dots l_N^{k_N}$ having the 
maximal number of distinct forms in the product. Since $\bm$ is not a product of all distinct forms, it should contain a
form $l_i$ in some power $k_i\ge 2$. Given that $k_i\ge 2$ and the degree of $\bm$ is $N-d-1$, one can find $d+2$ distinct forms $l_{i_1},\dots,l_{i_{d+2}}$ that do not appear in $\bm$. Since any $d+2$ of our forms span the dual space of $\RR^{d+1}$, we can express 
$l_i$ as a linear combination of $l_{i_1},\dots,l_{i_{d+2}}$. Now rewrite $\bm$ as 
$\left(\alpha_1\cdot l_{i_1}+\dots+\alpha_{d+2}\cdot l_{i_{d+2}}\right) l_1^{k_1}\dots l_i^{k_i-1}\dots l_N^{k_N}$, where 
$\alpha_1\cdot l_{i_1}+\dots+\alpha_{d+2}\cdot l_{i_{d+2}}=l_i.$ Thus we get an expression of $\bm$
as a linear combination of monomials $\alpha_j\cdot l_{i_{j}}l_1^{k_1}\dots l_i^{k_i-1}\dots l_N^{k_N}$, where each
such monomial has more distinct forms than $\bm$. Each of such monomials can be expressed as a linear combination of 
products of all distinct forms, since $\bm$ was chosen as a monomial with the maximal possible number of distinct 
forms, which cannot be expressed in such a way. This is a contradiction. Therefore $\bm$ can  also be expressed 
as a linear combination of products of all distinct forms.
\end{proof}

Below we consider the inverse problem for a weakly non-degenerate $S$, using 
notation from \eqref{eq:Lmatrix} and \eqref{eq:fullminor}. Here we no longer  have 
a natural basis of all simplices sharing a  common vertex $\bv_N$.  Because of that we need to consider all $N$ points and 
include  one more linear form $l_N$ into the corresponding matrix $\bL$. Slightly abusing our notation,
we denote by $\bL$  the same matrix as before, although it contains one more (last) column corresponding to $\bv_N$.  Similarly to  notation \eqref{eq:fullminor}, for a given set $J$ of $d+1$ linear 
forms, we denote by  $\bL(J)$  the determinant of the corresponding $(d+1)\times(d+1)$-minor of $\bL$. 

We introduce the extended $\binom{N-1}{d}\times\binom{N}{d+1}$-matrix $\widetilde{Mat}_{S}$ with columns consisting of the coefficients  of 
the homogeneous polynomial $l_{i_1}(\bu)\dots l_{i_{N-d-1}}(\bu)$ with respect to the monomial basis in the variables $(u_0,u_1,\dots,u_d).$  
By  Lemma~\ref{lem:generic},  $\widetilde{Mat}_{S}$ has full rank,  
since it determines a surjective linear map  {\em onto} $HPol(N-d-1,d+1)$. 
Thus $\widetilde{Mat}_{S}$ has a maximal minor with a non-vanishing determinant. 
Formula \eqref{eq:determinant} holds for the  determinant of any  maximal minor of  $\widetilde{Mat}_{S}$.

\begin{lemma}
\label{re:determinants}
Let $\cS$ be any set of $\binom{N-1}{d}$ columns of $\widetilde{Mat}_{S}$. We label 
each column $T\in\cS$ by the corresponding  subset  of the linear forms $l_1,\dots , l_N$ of cardinality $N-d-1$.  
Then the determinant of the maximal minor $\widetilde{Mat}_{S}(\cS)$ 
formed by the columns of $\cS$ is given by:  
\begin{equation}
\label{eq:determinant}
\det\left[\widetilde{Mat}_{S}(\cS)\right] = k(\cS)\cdot \prod_{\substack{J\in {N \brack d+1}:\\\forall T\in\cS~T\cap J\neq\emptyset}} \bL(J),
\end{equation}
where $k(\cS)$ is a constant (possibly equal to zero) depending only on the combinatorial structure of the $(N-d-1)$-tuples  
in the set $\cS$.
\end{lemma}

\begin{proof}
Fix the set $\cS$ as above. In what follows, we treat both sides of  \eqref{eq:determinant} as complex-valued polynomials  in 
$N\cdot (d+1)$ variables, these variables being the entries of  matrix $\bL$. 

We first show that every determinant $\bL(J)$ divides $\det\left[\widetilde{Mat}_{S}(\cS)\right]$. 
Indeed, let $J=\{j_1,\dots,j_{d+1}\}$ be a set of $(d+1)$ forms which has  a nonempty intersection with any 
$(N-d-1)$-tuple of  forms in $\cS$. Let $\z=(z_1,\dots,z_{N\cdot (d+1)})$ be a zero of the polynomial $\bL(J)$, which means
that forms $l_{j_1},\dots,l_{j_{d+1}}$ comprised of the corresponding coordinates of $\z$ 
are linearly dependent. 
Therefore, there is a non-zero vector $\bu_0\in\RR^{d+1}$, such that $l_{j_1}(\bu_0)=\dots=l_{j_{d+1}}(\bu_0)=0$. Consider the row vector $(\bu_0^{I})$ consisting of $\binom{N-1}{d}$ homogeneous monomials of degree $N-d-1$ evaluated at $\bu_0.$
We notice that $(\bu_0^{I})$ is in the kernel of $\widetilde{Mat}_{S}(\cS)$, as the product of $(\bu_0^{I})$ with each column 
vector $T\in\cS$ of $\widetilde{Mat}_{S}(\cS)$ is equal to $\prod_{j\in T}l_j(\bu_0)$; and every set $T\in\cS$ contains at least 
one of the forms $l_{j_1},\dots,l_{j_{d+1}}$ in such a product. Thus $\det\left[\widetilde{Mat}_{S}(\cS)\right]$  also vanishes at  
such $\z$. 

We recall a well-known fact (cf. e.g. \cite{Boch}*{Theorem~61.1}) that 
\[
\label{eq:detLJ}
\bL(J)=\det \bordermatrix{~ & l_{j_1} & l_{j_2} & \dots & l_{j_{d+1}}  \cr
u_0    & z_{j_1}   & z_{j_2}   & \cdots & z_{j_{d+1}}           \cr
u_1    & z_{j_1+N} & z_{j_2+N} & \cdots & z_{j_{d+1}+N}         \cr
\vdots & \vdots    & \vdots    & \ddots & \vdots                \cr
u_{d}  & z_{j_1+dN}& z_{j_2+dN}& \cdots & z_{j_{d+1}+dN}        \cr
},
\]
is an irreducible complex-valued polynomial in  variables $z_{j_1},z_{j_2}\dots,z_{j_{d+1}+Nd}$. 
Now if every zero of an irreducible polynomial $p(z_1,\dots,z_{N(d+1)})$ annihilates another polynomial 
$q(z_1,\dots,z_{N(d+1)})$,  then $p$ divides $q$. We conclude that $\bL(J)$
 divides $\det\left[\widetilde{Mat}_{S}(\cS)\right]$.

%
% A proof of the statement: p,q\in \C[x1,...xk] and if p(z)=0 => q(z)=0 and p(z)=0 is irreducible, then p|q.
% By Nullstellensatz we get that q^r = p * \C[x1,...xk] for some r. We also know that \C[x1,...,xk] is 
% a unique factorization domain. Thus we may conclude that p|q, because p is irreducible.  
%

Using the fact that each $\bL(J)$ is an irreducible polynomial and all $\bL(J)$'s are pairwise distinct (i.e., 
have distinct sets of projective zeros) 
we conclude that the product of $\bL(J)$'s in the right-hand side of \eqref{eq:determinant}
divides $\det\left[\widetilde{Mat}_{S}(\cS)\right].$

Finally, the product of $\bL(J)$'s has the degree 
\begin{equation}
\label{eq:degree_products}
(d+1)\left|\left.\left\{J\in {N \brack d+1} \right| \forall T\in\cS\quad T\cap J\neq\emptyset\right\}\right|.
\end{equation}
We observe that for each $T\in\cS$, the complementary set of $d+1$ forms cannot be taken as a feasible $J$.
We  notice further that these complements are the only exceptions for the choice of   $J$. Therefore  
as a feasible $J$ we can pick any of $\binom{N}{d+1}$ $(d+1)$-tuples except those $\binom{N-1}{d}$ 
complements of a $T\in\cS$. Therefore, \eqref{eq:degree_products} equals
\[
(d+1)\left(\binom{N}{d+1}-\binom{N-1}{d}\right)=(d+1)\binom{N-1}{d+1}=(N-d-1)\binom{N-1}{d}.
\]
The latter expression  coincides with  the degree of the polynomial $\det\left[\widetilde{Mat}_{S}(\cS)\right]$
(assuming that it is not a zero), as $\widetilde{Mat}_{S}(\cS)$ has $\binom{N-1}{d}$ columns and each entry is a homogeneous polynomial of degree $N-d-1$.

Hence $\widetilde{Mat}_{S}(\cS)$ coincides with  the product of $\bL(J)$'s up to a constant factor which might vanish. This constant does not depend on the entries of  matrix $\bL$
and hence it is completely determined by the set $\cS$, regardless of the location of points of $S$ in $\RR^{d}$.
% and corresponding coefficients of the linear forms $l_{j_k}$'s. 
\end{proof}

 Lemma~\ref{lem:generic} allows us to solve the inverse moment  problem  for a given weakly non-degenerate $S=\{\bv_1,\dots, \bv_N\}$ in  a certain linear space  $\widetilde{\M}(S)\supseteq\M^\De(S)$ of measures supported on $\conv(S)$. Namely, $\widetilde{\M}(S)$ is spanned by measures $\mu\in \widetilde{\M}(S)$ whose   normalized moment generating functions $F_\mu(\bu)$ belong to  $\Rat(S)$, 
i.e. $F_\mu(\bu)=P(\bu)/\Phi_S(\bu)$, where $\Phi_S(\bu)=\prod_{j=1}^N l_j(\bu)$ and $P(\bu)$ is a polynomial of degree at most $N-d-1$.  Indeed, by Lemma~\ref{lem:generic} any $R(\bu)\in \Rat(S)$ can be represented in the form
\begin{equation}\label{eq:R(u)}
R(\bu)=\sum_{i_1<i_2<\dots < i_{d+1}\le N}\frac{K_{i_1i_2\dots i_{d+1}}}{l_{i_1}l_{i_2}\dots l_{i_{d+1}}},
\end{equation}
with some real constants $K_{i_1i_2\dots i_{d+1}}$. If $\bv_{i_1}, \bv_{i_2},\dots , \bv_{i_{d+1}}$ span $\bR^d$ then the term 
$\frac{K_{i_1i_2\dots i_{d+1}}}{l_{i_1}l_{i_2}\dots l_{i_{d+1}}}$ can be interpreted as the normalized moment generating function of an appropriately scaled standard measure of the $d$-dimensional simplex spanned by these vertices. 

If $\bv_{i_1}, \bv_{i_2},\dots , \bv_{i_{d+1}}$ only span a hyperplane $H$ in $\bR^d$ then \eqref{eq:R(u)} 
corresponds to a singular (w.r.t. to the Lebesgue measure on $\bR^d$) measure $\mu_\delta$ supported on $\delta=\conv(\bv_{i_1},\dots, \bv_{i_{d+1}})$. 
One way to define it as the {weak limit} of a sequence of 
(absolutely continuous with respect to  the Lebesgue measure on $\bR^d$) measures---the  
appropriately scaled standard measures $\mu_{\delta_t}$ of family of $d$-dimensional simplices $\delta_t$ which 
degenerate into   
$\delta$ when $t=0$. There is no loss in generality in assuming
$K_{i_1i_2\dots i_{d+1}}=1$, i.e., to deal with probability measures.

\begin{proposition}\label{pr:degsimplex} Let $\bW=\{\bw_1,\dots, \bw_{d}, \bw_{d+1}\}$ be a $(d+1)$-tuple of points in $\bR^d$ such that $\bW$ spans a hyperplane $H\subset \bR^d$. Denote by $l_{\bw_1}=1-\lan \bw_1,\bu\ran,\dots, l_{\bw_{d+1}}=1-\lan \bw_{d+1},\bu\ran$  the associated linear forms. There exists a unique measure $\mu_{\bW}$ supported on $\delta=\conv({\bW})$ with the normalized moment generating function $F_{\mu_{\bW}}(\bu)$ given by   
\begin{equation}\label{eq:temp}
F_{\mu_{\bW}}(\bu)=\frac{1}{l_{\bw_1}l_{\bw_2}\dots l_{\bw_{d+1}}}.
\end{equation} 
\end{proposition}

\begin{proof}
Without loss of generality assume that $\bW=\{\bw_1,\dots,\bw_{d},  \bw_{d+1}\}$  is
ordered in such a way that $\{\bw_1,\dots, \bw_{d}\}$ span $H$. 
 Then, $\delta_t$ is defined as $\delta_t=\conv(\delta, \bw_{i_{d+1}}+t\z)$, with $\z$ a unit 
normal to $H$, and $\mu_{\delta_t}$ as the uniform density probability measure supported on $\delta_t$.
Then $\lim\limits_{t\to 0}\mu_{\delta_t}=\mu_\delta$, where $\lim$ is understood in sense
of \emph{weak convergence of distributions (measures)}, i.e., that $\lim\limits_{t\to 0}\int f d\mu_{\delta_t}=\int f d\mu_{\delta}$ for any
bounded, continuous real function on $\bR^d$, cf. e.g. \cite{MR1700749}.
Then, this measure has compact support, and thus is determined by its moments, cf. e.g. 
\cite{MR2244695}*{Proposition~3.2}. 
\end{proof}

\begin{rema+} One can prove that the integration of a smooth compactly supported function $\phi$ with respect to  the limiting measure $\mu_{\bW}$ is given by the integration of $\phi$ over $\delta$ with a continuous piecewise linear weight function uniquely determined by $\delta$.  Similar limits appear frequently in the theory of splines. Since we only need the existence of $\mu_{\bW}$ we do not pursue this topic here. 
\end{rema+} 

%Proposition~\ref{pr:degsimplex} clarifies the structure of  the space $\widetilde{\M}(S)$. We will comment on bases of this space later. Otherwise, 
Our solution of the inverse moment problem for the linear space $\widetilde{\M}(S)$ closely follows the pattern presented in Example~\ref{ex:two}. In other words, given a weakly non-degenerate $S$ and  the set of moments up to order $N-d-1$ we 
\begin{enumerate}[(i)]
\item produce  the rational function $R(\bu)\in \Rat(S)$ with 
Taylor coefficients coinciding with the normalized moments; 
\item represent $R(\bu)$ in the form \eqref{eq:R(u)}; 
\item for each term as in \eqref{eq:temp}, determine the underlying measure supported on the (probably degenerate) convex hull of the vertices $\bv_{i_1},\bv_{i_2},\dots,\bv_{i_{d+1}}$. 
\end{enumerate}
We can now prove our central result claiming that   $\M^\De(S)=\M(S)$ for a weakly non-degenerate $S$. 
\begin{proof}[Proof of Theorem~\ref{th:fund}] 
Theorem~\ref{th:fund} is already settled in Corollary~\ref{cor:generic}
for the case of strongly non-degenerate $S$.
It remains to consider the case of weakly
non-degenerate $S$.  The denominator of the moment generating function
$F_\cP(\bu)$ for an arbitrary generalized polytope $\cP$ with the vertex set
$S$ is of the form $\Pi_{i=1}^Nl_i$ by Proposition~\ref{prop:gener}, and its
numerator belongs to $Pol(N-d-1,d)$. As $S$ is weakly non-degenerate,
$F_\cP(\bu)$ can be written as a linear combination of the fractions as in \eqref{eq:R(u)},  where $(i_1,i_2,\ldots
i_{d+1})$ runs over the set of $(d+1)$-tuples of indices. If a $(d+1)$-tuple
$l_{i_1}, l_{i_2}, \ldots l_{i_{d+1}}$ is spanning then $\frac{K}
{l_{i_1}l_{i_2} \ldots l_{i_{d+1}}}$ is the moment generating function of the
measure supported on the simplex $\Delta$, determined by its denominator, with the 
uniform density
$K/d!\text{Vol}(\Delta)$. By Proposition~\ref{pr:degsimplex}, if a
$(d+1)$-tuple $l_{i_1}, l_{i_2}, \ldots l_{i_{d+1}}$ is not spanning then 
$\frac{K} {l_{i_1}l_{i_2} \ldots l_{i_{d+1}}}$ is the moment
generating function of a singular measure supported on a degenerate simplex.
As $\cP$ is a generalized polytope, its standard measure has no singular
components. Therefore, no degenerate simplices can appear in its decomposition.
\end{proof}
\begin{rema+} The latter proof demonstrates that if one starts from the set of moments of the 
standard measure $\mu$ of a polytope with the vertex set $S$ then %its $F_\mu(\bu)$ must lie in $\Rat(S)$ and 
we never obtain degenerate simplices while solving the inverse moment problem.
This is why $\M^\De(S)=\M(S)$. 
However, an explicit description of $\F^\De(S)$ for a general weakly non-degenerate $S$ is missing at
present.  For concrete  Examples~\ref{ex:three} and \ref{ex:four}, we give these descriptions below. 
\end{rema+}

Our final result computes $\dim \M^\De(S)$ and describes a procedure to construct 
a basis for  $\M^\De(S)$.

\begin{proposition}\label{pr:basic} 
Let $S=\{\bv_1,\ldots, \bv_N\}\subset \bR^d$ be an arbitrary weakly non-degene\-rate spanning set. Then 
\begin{enumerate}[(i)]
\item\label{it:i}
$\dim \M^\De(S) =\binom {N-1}{d}-\sharp_{deg}$ where $\sharp_{deg}$ is the number of degenerate simplices, i.e., the number of non-spanning $(d+1)$-tuples of points of $S$.
\item\label{it:relsi}
If $\delta$ is a degenerate $d$-dimensional simplex with vertices in $S\setminus\{\bv_i\}$ 
then there is exactly one linear dependence among the standard measures of all $d$-dimensional simplices 
on $\bv_i$ and $d$ vertices of $\delta$. 
\item\label{it:relsii}The standard measure of any $d$-dimensional simplex on $\bv_i$
is contained in at most one dependence as in (\ref{it:relsi}).
\item\label{it:ii}
For any vertex $\bv_i$, one can construct a (in general, non-unique)  basis $\B_i$ of $\M^\De(S)$ 
consisting of standard measures of $d$-dimensional simplices on $\bv_i$, as follows. 
\begin{enumerate}[(a)]
\item\label{it:a}
Start from the set $\B_i$ of the $d$-dimensional simplices on $\bv_i$.
\item \label{it:b}
For each degenerate simplex $\delta$ not containing $\bv_i$,
remove from $\B_i$ the standard measure of an arbitrary simplex on $\bv_i$ from the corresponding to $\delta$
linear dependence, cf. (\ref{it:i}).
\end{enumerate}
Thus we obtain $\binom {N-1}{d}-\sharp_{deg}$ standard measures of 
$d$-dimensional simplices, forming a basis of $\M^\De(S)$.
\end{enumerate}
\end{proposition}

\begin{proof} To prove (\ref{it:i}), notice that $\dim \widetilde{\M}(S)=\dim\Rat(S)=\binom{N-1}{d}$. 
As well, $\widetilde{\M}(S)=\M(S)\oplus \M_{deg}(S)$, 
where $\M_{deg}(S)$ is the linear span of the measures $\mu_\delta^{(1)}$ with $\delta$ 
running over the set of all degenerate simplices spanned by  $(d+1)$-tuples of dependent vertices in $S$, 
cf. Proposition~\ref{pr:degsimplex}. 
Observe that these measures $\mu_\delta^{(1)}$  are linearly independent,
as each degenerate simplex defines a singular measure supported in a proper hyperplane, and these
hyperplanes differ for different degenerate simplices.
We are done with (\ref{it:i}).

Let $\Sigma_0$ be a dependent $d+1$-subset of $S$, and 
$\delta=\conv(\Sigma_0)$ be as in  (\ref{it:relsi}). Then $\delta$ spans a hyperplane $H_0$.
As each $d$-dimensional simplex on $\sigma_0:=\bv_i$ and $d$ vertices from $\Sigma_0$ is uniquely
defined by the latter, it suffices to analyze dependencies between the standard measures of
$d-1$-simplices with vertices in $\Sigma_0$.
 
We can view $\Sigma_0$ as a weakly non-degenerate subset in $\bR^{d-1}\cong
H_0$.  By (\ref{it:i}), we have
$\dim\M^\De(\Sigma_0)=\binom{d}{d-1}-\sharp_{deg}(\Sigma_0)$.  If $\Sigma_0$ is
strongly non-degenerate as a subset of $H_0\cong\bR^{d-1}$, i.e.
$\sharp_{deg}(\Sigma_0)=0$, then $\dim\M^\De(\Sigma_0)=d$, i.e., there is
exactly one linear dependence between the standard measures of $d-1$-simplices
with vertices in $\Sigma_0$, and we are done.  Otherwise,
$\Sigma_0=\{\sigma_1\}\cup\Sigma_1$, with $\Sigma_1$ spanning a hyperplane
$H_1$ in $H_0$.  Moreover, this can only happen if $d\geq 3$. Now, we can
repeat the whole argument with $\sigma_1$ in place of $\sigma_0$, $\Sigma_1$ in place of
$\Sigma_0$, and $H_1$ in place of $H_0$.
Again, we either have $\Sigma_1$ strongly degenerate, and we are done, or we 
repeat this argument, etc., until
we hit a strongly non-degenerate $\Sigma_k$, which is bound to happen, as the
dimension goes down each iteration. This completes the proof of
(\ref{it:relsi}).

Then, (\ref{it:relsii}) stems from the fact that the vertices of
$d$-dimensional simplex on $\bv_i$ distinct from $\bv_i$ span a hyperplane, and the only possibility for
a degenerate simplex $\delta$ as in (\ref{it:relsi}) is to lie in this hyperplane.
 
Finally, to prove  (\ref{it:ii}), 
observe that the set $\B_i'$ of the standard measures of 
$d$-dimensional simplices containing a given vertex $\bv_i$ always spans $\M^\De(S)$, see Lemma~\ref{lm:simplex}.  
Now for each degenerate $d$-simplex $\delta$,   
we prune  $\B_i'$ by removing the standard measure of a simplex in the linear dependence corresponding to 
$\delta$. In view of (\ref{it:relsi}) and (\ref{it:relsii}) 
this process is well-defined and unambiguous. In the end we obtain
$\binom {N-1}{d}-\sharp_{deg}$ standard measures of 
$d$-dimensional simplices. In view of (\ref{it:i}) they form a basis of $\M^\De(S)$, as claimed. 
\end{proof}

\begin{remark}\label{rem:non-vanishing-minor} 
The above discussions show that the columns of $\widetilde{Mat}_S$ corresponding to degenerate simplices must necessarily be included in any non-vanishing maximal minor $\widetilde{Mat}_S(\cS)$. 
 \end{remark}

\medskip
We conclude our discussion of the weakly non-degenerate case with two examples. 

\begin{figure}[ht]
\begin{center}
\begin{picture}(360,150)(0,0)
\put(60,20){\line(1,0){100}}
%\put(220,20){\line(1,1){100}}
\put(60,20){\line(0,1){100}}
\put(60,20){\line(1,1){100}}
\put(160,20){\line(-1,1){100}}
%\put(120,20){\line(2,1){200}}
%\put(120,20){\line(1,2){100}}
%\put(120,120){\line(1,-1){100}}
\put(60,120){\line(5,0){100}}
\put(160,20){\line(0,5){100}}

\put(36,9){$\bv_5=(0,0)$}
\put(145,9){$\bv_2=(2,0)$}
\put(105,80){$\bv_1=(1,1)$}
\put(145,125){$\bv_3=(2,2)$}
\put(36,125){$\bv_4=(0,2)$}

\put(60,20){\circle*{3}}
\put(110,70){\circle*{3}}
\put(160,20){\circle*{3}}
\put(160,120){\circle*{3}}
%\put(220,220){\circle*{3}}
\put(60,120){\circle*{3}}

\put(260,20){\line(1,0){100}}
%\put(220,20){\line(1,1){100}}
\put(260,20){\line(0,1){100}}
%\put(120,120){\line(1,1){100}}
%\put(220,220){\line(1,-1){100}}
%\put(120,20){\line(2,1){200}}
%\put(120,20){\line(1,2){100}}
%\put(120,120){\line(1,-1){100}}
\put(260,120){\line(1,-1){100}}
%\put(360,20){\line(0,5){100}}

\put(236,9){$\bv_1=\bv_5=(0,0)$}
\put(345,9){$\bv_2=(2,0)$}
%\put(325,120){$\bv_2=(2,1)$}
\put(315,70){$\bv_3=(1,1)$}
\put(236,125){$\bv_4=(0,2)$}

\put(260,20){\circle*{3}}
\put(260,20){\circle{8}}
\put(360,20){\circle*{3}}
\put(310,70){\circle*{3}}
%\put(220,220){\circle*{3}}
\put(260,120){\circle*{3}}
\put(260,20){\line(1,1){50}}

\end{picture}
\end{center}
\caption{Vertices for Examples \ref{ex:three} and \ref{ex:four}.  \label{fg:T2}}
\end{figure}
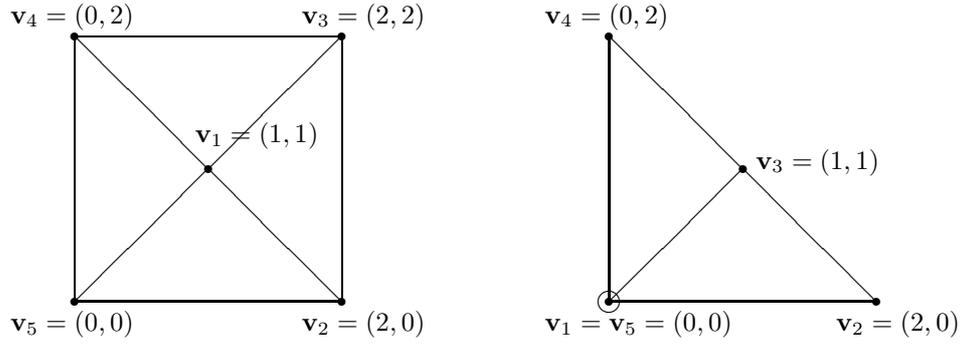

\begin{example}\label{ex:three}
Let $S=\{\bv_1,\bv_2, \bv_3,\bv_4,\bv_5\}$ where $\bv_1=(1,1), \bv_2=(2,0), \bv_3=(2,2), \bv_4=(0,2), \bv_5=(0,0)$. Then $l_1=1-u_1-u_2, l_2=1-2u_1, l_3=1-2u_1-2u_2, l_4=1-2u_2, l_5=1$. Calculating the products $l_il_j,\, i<j$ and taking their coefficients in the standard monomial basis of $Pol(2,2)$ we obtain the following $6\times 10$-matrix $\widetilde{Mat}_S$
$$\widetilde{Mat}_S=\bordermatrix{~ & l_1l_2 & l_1l_3 & l_1l_4 & l_1l_5&l_2l_3&l_2l_4&l_2l_5&l_3l_4&l_3l_5&l_4l_5 \cr 
    1        &1&1&1&1&1&1&1&1&1&1\cr
u_1       &-3&-3&-1&-1&-4&-2&-2&-2&-2&0&\cr
u_2       &-1 &-3&-3&-1&-2&-2&0&-4&-2&-2\cr
u_1^2   &2&2&0&0&4&0&0&0&0&0\cr
u_1u_2&2&4&2&0&4&4&0&4&0&0\cr
u_2^2&0&2&2&0&0&0&0&4&0&0}. $$
Its rank equals $6$ and one of  non-vanishing maximal minors consists of the columns with numbers $\cS=\{5,6,7,8,9,10\}$. (Recall that any non-vanishing maximal minor must include columns 6 and 9 corresponding to degenerate triples $(\bv_1,\bv_3,\bv_5)$ and $(\bv_1,\bv_2,\bv_4)$ resp.)  The corresponding submatrix $\widetilde{Mat}_S(\cS)$ equals 
$$\widetilde{Mat}_S(\cS)=\bordermatrix{~ & l_2l_3  & l_2l_4 & l_2l_5&l_3l_4&l_3l_5&l_4l_5\cr 
1&1&1&1&1&1&1\cr
u_1&-4&-2&-2&-2&-2&0\cr
u_2&-2&-2&0&-4&-2&-2&\cr
u_1^2&4&0&0&0&0&0\cr
u_1u_2&4&4&0&4&0&0\cr
u_2^2&0&0&0&4&0&0}. $$
Further, 
$$4\widetilde{Mat}_S^{-1}(\cS)=\bordermatrix{~ & 1  & u_1 & u_2&u_1^2&u_1u_2&u_2^2\cr 
l_2l_3&0&0&4&0&-4&4\cr
l_2l_4&0&0&0&0&-2&2\cr
l_2l_5&0&0&2&0&-2&0&\cr
l_3l_4&1&-1&0&0&-1&1\cr
l_3l_5&0&1&0&0&-1&0&\cr
l_4l_5&0&-1&1&1&-1&0}. $$
Thus, given an arbitrary  rational function $R(u_1,u_2)=P(u_1,u_2)/\Phi_S(u_1,u_2)$, where  
$P(u_1,u_2)=a_{00}+a_{1,0}u_1+a_{0,1}u_2+a_{2,0}u_1^2+a_{11}u_1u_2+a_{02}u_2^2$ is any 
polynomial  of degree at most $2$ and $\Phi_S(u_1,u_2)=l_1l_2l_3l_4l_5$, we obtain 
$$\begin{cases}
w_{145}=-a_{10}-a_{20}-a_{11}\\
w_{135}=-\frac{1}{2}(a_{11}-a_{02})\\
w_{134}=\frac{1}{2}(a_{01}-a_{11})\\
w_{125}=\frac{1}{4}(a_{00}-a_{10}-a_{11}+a_{02})\\
w_{124}=\frac{1}{4}(a_{10}-a_{11})\\
w_{123}=-\frac{1}{4}(a_{10}-a_{01}-a_{20}+a_{11}).\\
\end{cases}
$$
Triangles $\De_{135}$ and $ \De_{124}$ are degenerate which implies that  
if the original measure we are recovering is polygonal then $w_{135}=w_{124}=0$. 
Therefore the linear space  of numerators $P(u_1,u_2)$ for the space 
$\F^\De(S)$ in this example is given by the relation $$a_{01}=a_{11}=a_{02}.$$ 
\end{example}

Our last example is  more degenerate than the previous one, although still  weakly non-degenerate. In fact, in this example $S$ is a multiset since 
$\bv_1=\bv_5$. It shows that our technique can be generalized to a certain class of multisets as well.  
\begin{example}\label{ex:four}
Set $S=\{\bv_1,\bv_2, \bv_3,\bv_4,\bv_5\}$, where $\bv_1=\bv_5=(0,0), \bv_2=(2,0), \bv_3=(1,1), \bv_4=(0,2)$. Then $l_1=l_5=1, l_2=1-2u_1, l_3=1-u_1-u_2, l_4=1-2u_2$. Calculating all products $l_il_j,\, i<j$ and taking their coefficients in the standard monomial basis of $Pol(2,2)$, we obtain the following $6\times 10$-matrix $\widetilde{Mat}_S$:
$$\widetilde{Mat}_S=\bordermatrix{~ & l_1l_2 & l_1l_3 & l_1l_4 & l_1l_5&l_2l_3&l_2l_4&l_2l_5&l_3l_4&l_3l_5&l_4l_5 \cr 
1&1&1&1&1&1&1&1&1&1&1\cr
u_1&-2&-1&0&0&-3&-2&0&-1&-1&0\cr
u_2&0&-1&-2&0&-1&-2&-2&-3&-1&-2\cr
u_1^2&0&0&0&0&2&0&0&0&0&0\cr
u_1u_2&0&0&0&0&2&4&0&2&0&0\cr
u_2^2&0&0&0&0&0&0&0&2&0&0}.$$
Its rank equals $6$ and a non-vanishing maximal minor consists of the columns with numbers $\cS=\{1,3,4,5,6,8\}$. The corresponding submatrix $\widetilde{Mat}_S(\cS)$ equals 
$$\widetilde{Mat}_S(\cS)=\bordermatrix{~ & l_1l_2  & l_1l_4 & l_1l_5&l_2l_3&l_2l_4&l_3l_4\cr 
1&1&1&1&1&1&1\cr
u_1&-2&0&0&-3&-2&-1\cr
u_2&0&-2&0&-1&-2&-3&\cr
u_1^2&0&0&0&2&0&0\cr
u_1u_2&0&0&0&2&4&2\cr
u_2^2&0&0&0&0&0&2}. $$
Further, 
$$4\widetilde{Mat}_S^{-1}(\cS)=\bordermatrix{~ & 1  & u_1 & u_2&u_1^2&u_1u_2&u_2^2\cr 
l_1l_2&0&-2&0&-2&-1&0\cr
l_1l_4&0&0&-2&0&-1&-2\cr
l_1l_5&4&2&2&1&1&1&\cr
l_2l_3&0&0&0&2&0&0\cr
l_2l_4&0&0&0&-1&1&-1&\cr
l_3l_4&0&0&0&0&0&2}.$$
Thus, given an arbitrary  rational function $R(u_1,u_2)=P(u_1,u_2)/\Phi_S(u_1,u_2)$, where  
$P(u_1,u_2)=a_{00}+a_{1,0}u_1+a_{0,1}u_2+a_{2,0}u_1^2+a_{11}u_1u_2+a_{02}u_2^2$ is any 
polynomial  of degree at most $2$ and $\Phi_S(u_1,u_2)=l_1l_2l_3l_4l_5$, we obtain 
$$\begin{cases}
w_{345}=-\frac{1}{4}(2a_{10}+2a_{20}+a_{11})\\
w_{235}=-\frac{1}{4}((2a_{01}+a_{11}+2a_{02})\\
w_{234}=4a_{00}+2a_{01}+2a_{10}+a_{20}+a_{11}+a_{02}\\
w_{145}=2a_{20}\\
w_{135}=-a_{10}+a_{11}-a_{20}\\
w_{125}=2a_{02}.\\
\end{cases}
$$
Notice that  triangles $\De_{125}, \De_{135}, \De_{145}, \De_{234}$ are degenerate. If we know that the original measure we are  recovering is polygonal then one should get $w_{125}=w_{135}=w_{145}=w_{234}=0$.   Therefore, the linear space of numerators $P(u_1,u_2)$ for the space 
$\F^\De(S)$ in this example is given by the system of equations:  
$$\begin{cases}
a_{20}=0\\
a_{10}=a_{11}\\
a_{02}=0\\
4a_{00}+2a_{01}+3a_{10}=0.\\
\end{cases}
$$
\end{example}
%Observe that in both Examples \ref{ex:three} and \ref{ex:four}  the rows corresponding to the degenerate triangles must necessarily be included in any basis of the column space. 

\section{Remarks and open problems}\label{s5}
%{\bf 1.}
\begin{rema+}
A weaker form of Corollary~\ref{cor:arbit} (i.e., the rationality of
$F_{\cP}^\rho(\U)$, but without the claim on the particular
shape of the denominator) can be derived directly from \eqref{eq:Mainint}
by using Stokes formula, along the lines of \cite{Bar2}*{Lemma~1}.
\end{rema+}

\begin{problem}
%{\bf 2.}  {\bf Problem 2.}  
Find an appropriate version of Theorem~\ref{th:weight1},
applicable to non-simple and/or non-convex polytopes. 
\end{problem}

%{\bf 3.} 
\begin{rema+}
%Let us introduce a standard $\bZ$-lattice in $\M^\Delta(S)$.
Choose an arbitrary basis $\{\Delta_j\}$
of $\M^\Delta(S)$ consisting of the standard measures of simplices.  
The set $\{\Delta_j\}$
spans an integer lattice in $\M^\Delta(S)$. 
(One can easily see that this lattice is invariantly defined independently of the choice of a 
basis of standard measures of simplices.) Denote by $\M_\bZ^\Delta(S)$ the space $\M^\Delta(S)$ with the latter lattice. We can prove the following. 
\begin{proposition}\label{prop:rat}
Any generalized polytope $\cP\in \cP(S)$ with  standard measure $\mu_\cP$
corresponds to a rational point in  $\M^\Delta_\bZ(S)$.
\end{proposition}

\begin{proof} 
As  Conjecture~\ref{conj:main}  follows from \cite{ABR17}*{Theorem~1},
one has $\mu_\cP\in\M^\Delta(S)$.
One can easily show that $\cP$ can be 
represented as the union of the closures of connected components of 
$\bR^d\setminus H(S)$, where $H(S)$ is the hyperplane arrangement consisting of all hyperplanes spanned by 
$d$-tuples of points in $S$. (The converse is obviously not true.) 
Let $\tilde S\supseteq S$ be the {\it extended set of vertices} obtained by adding to $S$ 
all vertices of the hyperplane arrangement $H(S)$. Since each connected component in 
$\bR^d\setminus H(S)$ is convex, it can be triangulated on $\tilde S$. 
Consider the space $\M^\Delta_\bZ(\tilde S)$. Obviously, 
$\mu_\cP$ is an integer point in   $\M^\Delta_\bZ(\tilde S)$.  
Also, $\M^\Delta_\bZ(S)$ is contained in $\M^\Delta_\bZ(\tilde S)$ as a sublattice.  
Thus, if $\mu_\cP$ belongs to $\M^\Delta_\bZ(S)$ it is a rational point there. 
\end{proof}
\end{rema+}

\begin{problem}
One can also define a rational convex cone $\mathfrak{Pos}(S)\subset \M_\bZ(S)$ by taking non-negative linear combinations of all $\mu_\cP$, where $\cP$ runs over the set of all generalized polytopes in $\cP(S)$.  
 \begin{conjecture}\label{conj:cone} The rational cone $\mathfrak{Pos}(S)$ is uniquely determined by the oriented matroid associated to $S$. 
 \end{conjecture}
\end{problem}

We conclude this section with the following question. 
\begin{problem}\label{prob:ext} Is it possible to describe the extremal rays of $\mathfrak{Pos}(S)$? 
\end{problem}
One can easily show  that a simplex from $\cP(S)$ spans an extremal ray of $\mathfrak{Pos}(S)$ if and only if it does not contain any points of $S$ distinct from its vertices.  
Problem~\ref{prob:ext} is apparently closely related to the problem of classification of combinatorial types of point arrangements, see e.g. \cite{FuMiMo} and references therein.  

\subsection{Hyperplane arrangements and Laplace transform}\label{subs:arr}
After the first version \cite{GPSS12}
of this text was released, it was pointed out to us
that Laplace transform technique developed for studying hyperplane
arrangements in \cite{BrVe99} (see also \cite{DeCoPro11}*{Sect.~II.8}) and
the corresponding knowledge accummulated simplify and strengthen a number of
our results. Here we sketch the key ideas, leaving full details for another
publication.

Let $\cP$ be a generalized polytope.
It is natural to homogenize its normalized moment
generating function---the rational function $F^\rho_\cP(\bu)$
from Proposition~\ref{prop:gener}---so that its numerator and
denominator become homogeneous, using an extra variable $\bu_0$. E.g. this allows to account for the
origin appearing as a vertex of $\cP$. It also has a natural interpretation
in terms of measures. Namely, embed $\cP$ in the hyperplane $\{\bu \mid \bu_0=1\}$, and consider the cone spanned by $\cP$;
equip this cone with exponentially decaying in the direction $\bu_0$ measure.
Then the Laplace transform of this measure 
is the homogenization $\tilde{F}_\cP(\tilde{\bu})$ of $F^\rho_\cP(\bu)$,
where we denoted $\tilde{\bu}:=(\bu_0,\bu_1,\dots,\bu_d)$.

Note that the denominator of $\tilde{F}_\cP(\tilde{\bu})$ is the product
of powers of linear forms $\ell_\bv(\tilde{\bu})$, with $\bv\in\Verti(\cP)$.
The paper \cite{BrVe99} associates to the hyperplane arrangement
specified by the corresponding hyperplanes the algebra of rational functions
generated by the reciprocials of the $\ell_\bv$, endowed
with the natural action of the polynomial differential operators.
Then it proceeds to show that
$\tilde{F}_\cP(\tilde{\bu})$ admits a decomposition into a sum of
$\tilde{F}_\Delta(\tilde{\bu})$, with $\Delta$ ranging through $d$-simplices
with vertices in $\Verti(\cP)$,
whenever $\tilde{F}(\tilde{\bu})$ corresponds to a non-singular
polynomial density measure.
This in particular implies
Conjecture~\ref{conj:main}, and much more.

\bibliography{poly_gfunc,poly}
\end{document}